\documentclass[11pt]{article}
\usepackage{amssymb,amsmath,color,amsfonts,float,amscd,amsthm,bm,mathrsfs}

\allowdisplaybreaks[4]
\catcode`!=11
\let\!int\int \def\int{\displaystyle\!int}
\let\!lim\lim \def\lim{\displaystyle\!lim}
\let\!sum\sum \def\sum{\displaystyle\!sum}
\let\!sup\sup \def\sup{\displaystyle\!sup}
\let\!inf\inf \def\inf{\displaystyle\!inf}
\let\!cap\cap \def\cap{\displaystyle\!cap}
\let\!max\max \def\max{\displaystyle\!max}
\let\!min\min \def\min{\displaystyle\!min}

\let\oldsection\section
\renewcommand\section{\setcounter{equation}{0}\oldsection}

\newtheorem{theorem}{Theorem}[section]
\newtheorem{lemma}[theorem]{Lemma}

\newtheorem{proposition}[theorem]{Proposition}

\theoremstyle{definition}

\theoremstyle{remark}
\newtheorem{remark}[theorem]{Remark}

\newcommand{\norm}[1]{\left\|#1\right\|}
\newcommand{\abs}[1]{\left|#1\right|}

\usepackage[colorlinks=true, linkcolor=blue, citecolor=blue, urlcolor=blue, breaklinks=true]{hyperref}
\hypersetup{
    pdftitle={Asymptotic behavior of solutions of the compressible MHD equations on the half space},
    pdfsubject={Asymptotic behavior of solutions of the compressible MHD equations on the half space},
    pdfkeywords={Compressible MHD Equations,
Initial Boundary Value Problem,
Asymptotic Behavior of Solutions,
Insulating Boundary Condition,
Perfectly Conducting Boundary Condition}
}

\begin{document}

\title{Asymptotic behavior of solutions of the compressible MHD equations on the half space}
\author{Song Gao\footnote{Corresponding author. School of Mathematical Sciences, Shanghai Jiao Tong University, Shanghai 200240, P. R. China. Email: gaoshanbudi@sjtu.edu.cn}
}
\date{}
\maketitle

\begin{abstract}
In this paper, we study the global well-posedness and large-time asymptotic behavior of strong solutions to the $n$-dimensional ($n\ge 2$) compressible viscous resistive magnetohydrodynamics (MHD) equations in the upper half-space $\mathbb{R}_+^n$. We impose the standard no-slip boundary condition on the velocity field and two types of physically relevant boundary conditions on the magnetic field: the insulating boundary condition and the perfectly conducting boundary condition. Under the assumption that the initial data are small perturbations of the constant equilibrium state $(\bar{\rho}, \mathbf{0}, \mathbf{0})$, we first establish the global existence and uniqueness of strong solutions in Sobolev spaces by deriving uniform energy estimates and using a bootstrapping argument. Then, by combining the sharp decay estimates of solutions to the corresponding linearized system with Duhamel's principle, we obtain the  decay rates of solutions to the nonlinear problem. Specifically, we prove that the solutions decay at the rate $t^{-\frac{n}{2}\left(1-\frac{1}{p}\right)}$ in $L^p$ norm for $2\le p\le\infty$, and we also establish the decay rates of their first-order derivatives. Furthermore, we show that the global solutions converge to the linearized solutions with an improved decay rate as time goes to infinity. Our results extend the classical results of Kagei and Kobayashi \cite{Kagei-Kobayashi2002, Kagei-Kobayashi2005} for the compressible Navier--Stokes equations on the half space to the compressible MHD system with both insulating and perfectly conducting boundary conditions.
\end{abstract}


\tableofcontents

\section{Introduction and main results}\label{sec:1}
The present work investigates both the global well-posedness and the large-time asymptotic behavior of strong solutions to the n-dimensional ( \(n\ge 2\) ) compressible viscous resistive magnetohydrodynamics (MHD) equations, which are formulated on the upper half-space
\(\mathbb{R}_+^n \triangleq \{(x', x_n) \,\big|\, x'\triangleq(x_1,\dots,x_{n-1})\in\mathbb{R}^{n-1},\ x_n \ge 0\}\).
\begin{equation}\label{1.1}
 \left\{
\begin{aligned}
&\partial_t\rho + \operatorname{div}(\rho\mathbf{v}) = 0,\\
&\rho\left(\partial_t\mathbf{v} + \mathbf{v} \cdot \nabla \mathbf{v}\right) - \mu\Delta\mathbf{v} - (\mu + \lambda)\nabla \operatorname{div} \mathbf{v} + \nabla P =\mathbf{B}\cdot \nabla \mathbf{B}-\frac{1}{2}\nabla\abs{\mathbf{B}}^2, \\
&\partial_t\mathbf{B} - \kappa\Delta\mathbf{B} + \mathbf{v} \cdot \nabla \mathbf{B} + \mathbf{B}\operatorname{div} \mathbf{v} = \mathbf{B} \cdot \nabla \mathbf{v}, \\
&\operatorname{div} \mathbf{B} = 0, \\
&(\rho, \mathbf{v}, \mathbf{B})\big|_{t=0} = (\rho_0, \mathbf{v}_0, \mathbf{B}_0).
\end{aligned}
\right.
\end{equation}
In the above system, the dependent variables \(\rho \in \mathbb{R}_+\), \(\mathbf{v} \in \mathbb{R}^n\) and \(\mathbf{B} \in \mathbb{R}^n\) stand for the fluid density, velocity field and magnetic field, respectively.
The pressure \(P = P(\rho)\) is a smooth function depending on the density, which satisfies \(P'(\rho) > 0\) and takes the value \(P'(\bar{\rho}) = 1\) at the constant reference density \(\bar{\rho} > 0\).
The positive parameter \(\kappa > 0\) characterizes the electrical conductivity associated with the magnetic field evolution.
The viscosity coefficients \(\mu\) and \(\lambda\) fulfill the constraints \(\mu > 0\) and \(2\mu + n\lambda \geq 0\).
The divergence-free condition \(\operatorname{div} \mathbf{B}=0\) is automatically preserved for all time as a direct consequence of the initial condition \(\operatorname{div} \mathbf{B}_0=0\).

The velocity field satisfies the standard no-slip boundary condition at the boundary of the half-space:
\begin{equation}\label{1.2}
\mathbf{v}\big|_{x_n=0} = \mathbf{0}.
\end{equation}

For the magnetic field, we adopt two physically relevant boundary regimes: the insulating boundary condition and the perfectly conducting boundary condition, given by
\begin{equation}\label{1.3}
\left(\mathbf{B}', \partial_{x_n} B_n\right)\big|_{x_n=0} = \mathbf{0}\quad \text{or}\quad\left(\partial_{x_n}\mathbf{B}',B_n\right)\big|_{x_n=0} = \mathbf{0},
\end{equation}
where the prime symbol \('\) stands for the tangential components of an \(n\)-dimensional vector field.

This paper is primarily devoted to investigating the global well-posedness and the large-time asymptotic behavior of strong solutions to systems \eqref{1.1}–\eqref{1.3}, in a small neighborhood of the constant equilibrium state \((\bar{\rho}, \mathbf{0}, \mathbf{0})\).

Without loss of generality, we normalize the parameters by setting \(\bar{\rho} = \kappa = 1\), and we let \(\alpha = \mu > 0\) and \(\beta = \mu + \lambda \geq 0\). We then introduce the density perturbation given by
\begin{align*}
a \triangleq \rho - 1.
\end{align*}

Under this change of variables, the original initial-boundary value problem \eqref{1.1}–\eqref{1.3} can be recast as
\begin{equation}\label{1.4}
 \left\{
\begin{aligned}
&\partial_t a + \operatorname{div} \mathbf{v} = f_1,\\
&\partial_t\mathbf{v} -  \alpha \Delta\mathbf{v} -  \beta \nabla \operatorname{div} \mathbf{v} + \nabla a = {f}_2, \\
&\partial_t\mathbf{B} - \Delta \mathbf{B} = {f}_3, \\
&\operatorname{div} \mathbf{B} = 0, \\
&(a, \mathbf{v}, \mathbf{B})\big|_{t=0} = (a_0, \mathbf{v}_0, \mathbf{B}_0),\\
&\mathbf{v}\big|_{x_n=0} = \mathbf{0},\\
&\left(\mathbf{B}', \partial_{x_n} B_n\right)\big|_{x_n=0} = \mathbf{0}\quad \text{or}\quad\left(\partial_{x_n}\mathbf{B}',B_n\right)\big|_{x_n=0} = \mathbf{0},
\end{aligned}
\right.
\end{equation}
where \(f_i\) \((i=1,2,3)\) are nonlinear terms given by
\begin{equation}\notag
 \left\{
\begin{aligned}
&f_1 \triangleq -\mathbf{v} \cdot \nabla a - a\operatorname{div} \mathbf{v},\\
&{f}_2 \triangleq -\mathbf{v} \cdot \nabla \mathbf{v} + \mathbf{B} \cdot \nabla \mathbf{B} - \mathbf{B}\nabla \mathbf{B} - J(a)\nabla a \\
&\quad\quad - I(a)\left( \alpha \Delta\mathbf{v} +  \beta \nabla \operatorname{div} \mathbf{v} + \mathbf{B} \cdot \nabla \mathbf{B} - \mathbf{B}\nabla \mathbf{B}\right),\\
&{f}_3 \triangleq -\mathbf{v} \cdot \nabla \mathbf{B} - \mathbf{B}\operatorname{div} \mathbf{v} + \mathbf{B} \cdot \nabla \mathbf{v}
\end{aligned}
\right.
\end{equation}
with
\begin{align*}
I(a) \triangleq \frac{a}{a + 1} \quad \text{and} \quad J(a) \triangleq \frac{P'(a + 1)}{a + 1} - 1.
\end{align*}
In the subsequent sections, we establish global uniform energy estimates for solutions to the reformulated system \eqref{1.4}.
To further facilitate the decay-rate analysis carried out later, we also introduce an equivalent perturbed formulation expressed in terms of the momentum variable. Specifically, we define
\[
\mathbf{m}\triangleq(a+1)\mathbf{v},
\]
with the aid of which the initial-boundary value problem \eqref{1.4} can be rewritten into the following equivalent form:
\begin{equation}\label{1.5}
 \left\{
\begin{aligned}
&\partial_t a + \operatorname{div} \mathbf{m} = 0,\\
&\partial_t\mathbf{m} -  \alpha \Delta\mathbf{m} -  \beta \nabla \operatorname{div} \mathbf{m} + \nabla a = \operatorname{div} \mathcal{N}_1+\operatorname{div} \mathcal{N}_2, \\
&\partial_t\mathbf{B} - \Delta \mathbf{B} = \operatorname{div} \mathcal{N}_3, \\
&(a, \mathbf{m}, \mathbf{B})\big|_{t=0} = (a_0, \mathbf{m}_0, \mathbf{B}_0),\\
&\mathbf{m}\big|_{x_n=0} = \mathbf{0},\\
&\left(\mathbf{B}', \partial_{x_n} B_n\right)\big|_{x_n=0} = \mathbf{0}\quad \text{or}\quad\left(\partial_{x_n}\mathbf{B}',B_n\right)\big|_{x_n=0} = \mathbf{0},
\end{aligned}
\right.
\end{equation}
Here \(\mathbf{m} \in \mathbb{R}^n\) stands for the momentum variable, with initial datum \(\mathbf{m}_0\triangleq(a_0+1)\mathbf{v}_0\).
For each \(i=1,2,3\), \(\mathcal{N}_i\) denotes an \(n\times n\) matrix with entries \(\left((\mathcal{N}_i)_{j k}\right)_{1 \leq j, k \leq n}\).
The explicit forms of these matrices are given by
\enlargethispage{1\baselineskip} 
\pagebreak[1] 
\begin{equation}\notag
\left\{
\begin{aligned}
(\mathcal{N}_1)_{jk} &\triangleq -\delta_{jk}\left(a^2 \int_{0}^{1} (1-\theta) P''\left(a\theta + 1\right) \mathrm{d}\theta\right) -  \alpha  \partial_{k}\left( \frac{a m_j}{a+1} \right) \\
&\quad -\delta_{jk} \beta  \operatorname{div}\left( \frac{a \mathbf{m}}{a+1} \right) - \frac{m_j m_k}{a+1}, \\
(\mathcal{N}_2)_{jk} &\triangleq B_j B_k - \frac{1}{2}\delta_{jk}\abs{\mathbf{B}}^2, \\
(\mathcal{N}_3)_{jk} &\triangleq \frac{1}{{a}+1}\left(m_j B_k-B_j m_k\right),
\end{aligned}
\right.
\end{equation}
Here, for every \(i = 1, 2, 3\), the \(j\)-th component of \(\operatorname{div} \mathcal{N}_i\), obtained by taking the divergence along each row of the matrix, is expressed as \(\sum_{k=1}^{n} \partial_{k} (\mathcal{N}_i)_{j k}\).

The linearized system corresponding to \eqref{1.5} reads:
\begin{equation}\label{1.6}
 \left\{
\begin{aligned}
&\partial_t a + \operatorname{div} \mathbf{m} = 0,\\
&\partial_t\mathbf{m} -  \alpha \Delta\mathbf{m} -  \beta \nabla \operatorname{div} \mathbf{m} + \nabla a = \mathbf{0}, \\
&\partial_t\mathbf{B} - \Delta \mathbf{B} = \mathbf{0}, \\
&(a, \mathbf{m}, \mathbf{B})\big|_{t=0} = (a_0, \mathbf{m}_0, \mathbf{B}_0),\\
&\mathbf{m}\big|_{x_n=0} = \mathbf{0},\\
&\left(\mathbf{B}', \partial_{x_n} B_n\right)\big|_{x_n=0} = \mathbf{0}\quad \text{or}\quad\left(\partial_{x_n}\mathbf{B}',B_n\right)\big|_{x_n=0} = \mathbf{0},
\end{aligned}
\right.
\end{equation}

This system decouples fully into two independent subsystems, each of which can be analyzed separately:

\begin{itemize}
\item[(i)] The density perturbation \(a\) and momentum perturbation \(\mathbf{m}\) satisfy the following linearized compressible Navier–Stokes system subject to the no-slip boundary condition:
\begin{equation}\label{1.7}
\left\{
\begin{aligned}
&\partial_t a + \operatorname{div} \mathbf{m} = 0,\\
&\partial_t\mathbf{m} - \alpha \Delta\mathbf{m} - \beta \nabla \operatorname{div} \mathbf{m} + \nabla a = \mathbf{0}, \\
&(a, \mathbf{m})\big|_{t=0} = (a_0, \mathbf{m}_0),\\
&\mathbf{m}\big|_{x_n=0} = \mathbf{0}.
\end{aligned}
\right.
\end{equation}

\item[(ii)] The magnetic field perturbation \(\mathbf{B}\) satisfies the following vector-valued heat equation, subject to either the insulating or perfectly conducting boundary condition defined above:
\begin{equation}\label{1.8}
\left\{
\begin{aligned}
&\partial_t\mathbf{B} - \Delta \mathbf{B} = \mathbf{0}, \\
&\mathbf{B}\big|_{t=0} = \mathbf{B}_0,\\
&\left(\mathbf{B}', \partial_{x_n} B_n\right)\big|_{x_n=0} = \mathbf{0}\quad \text{or}\quad\left(\partial_{x_n}\mathbf{B}',B_n\right)\big|_{x_n=0} = \mathbf{0}.
\end{aligned}
\right.
\end{equation}
\end{itemize}
We write \((a,\mathbf{m}) \triangleq \mathscr{U}(t)\mathbf{u}_0\) and \(\mathbf{B} \triangleq \mathscr{B}(t)\mathbf{B}_0\) for the solutions to subsystems \eqref{1.7} and \eqref{1.8} with initial data \(\mathbf{u}_0 \triangleq (a_0,\mathbf{m}_0)\) and \(\mathbf{B}_0\), respectively.

\textbf{Related works}: The local well-posedness theory for the MHD equations can be established using methods analogous to those introduced in \cite{MN1982,Tani1977}. Over the past decades, a wealth of important results on global solutions to the Navier–Stokes and MHD systems have been obtained; a representative but non-exhaustive selection of references includes \cite{chen2002global,chen2003existence,chen2010global,ducomet2006equations,hu2008compactness,hu2008global,hu2010global,jiang2003global,kawashima1981global,kawashima1982smooth,kobayashi2002some,
matsumura1980initial,wang2003large,zhang2008global,zhu2010asymptotic}.

Decay properties of solutions represent a topic of both theoretical interest and practical significance in the mathematical analysis of the Navier–Stokes and MHD equations. Substantial progress has been made in recent years regarding decay estimates for smooth solutions, both for convergence toward constant equilibrium states and for nontrivial stationary profiles in the presence of external potential forces. For a comprehensive review of this line of research, interested readers may wish to consult the works cited in \cite{chen2010global,duan2007optimalLpLq,duan2007optimal,gao2016optimal,hoff1995multi,hoff1997pointwise,Kagei-Kobayashi2002,Kagei-Kobayashi2005,kobayashi1999decay,
kobayashi2002remarks,kobayashi2002some,
liu1998pointwise,matsumura1979initial,ponce1985global,Pu-Guo2013,ukai2006convergence,wang1997large}.
Turning specifically to compressible fluid dynamics, Matsumura and Nishida \cite{Matsumura-Nishida1983} were the first to establish the unique existence of global \(H^3\) solutions to the compressible Navier–Stokes equations in the half-space \(\mathbb{R}_{+}^{3}\) under sufficiently small initial perturbations. Building on this foundational work, Kagei and Kobayashi made landmark contributions in two seminal papers \cite{Kagei-Kobayashi2002, Kagei-Kobayashi2005}, in which they systematically investigated the large-time asymptotic behavior of compressible Navier–Stokes solutions under the no-slip velocity boundary condition. For three-dimensional compressible viscous flows in exterior domains, decay rate estimates were derived in \cite{kobayashi1999decay,kobayashi2002some} via linearization combined with Fourier analytic techniques. Duan et al. \cite{duan2007optimalLpLq,duan2007optimal} subsequently established sharp decay rates for the compressible Navier–Stokes equations coupled with an energy equation in the presence of potential body forces. The techniques they developed were later applied to the isentropic compressible MHD equations in \cite{chen2010global}, yielding corresponding decay estimates. More recently, Pu and Guo \cite{Pu-Guo2013} proved the global existence of smooth solutions and established their optimal decay rates for the three-dimensional full compressible MHD equations under small initial perturbations. Gao et al. \cite{gao2016optimal} further obtained sharp decay bounds for high-order derivatives of solutions to the full compressible MHD system using the Fourier splitting method, again under the assumption of small initial data. For related results concerning cylindrically symmetric compressible Navier–Stokes and MHD equations, we refer the interested readers to \cite{fan2007zero,fan2007vanishing,frid2000vanishing,ukai2006convergence,zhai2025nodea,zhai2025analappl} for further details.

In this paper, we aim to extend the results of \cite{Kagei-Kobayashi2002, Kagei-Kobayashi2005} to the n-dimensional (\(n\ge 2\)) compressible viscous resistive MHD system in the half-space \(\mathbb{R}_+^n\), subject to the no-slip velocity boundary condition together with either insulating or perfectly conducting boundary conditions for the magnetic field.

\textbf{Notation}. For ease of notation, we write \(L^p(\mathbb{R}_+^n)\) for \(p\geq 1\) and \(H^k(\mathbb{R}_+^n)\) for \(k\geq 0\) to denote the standard Lebesgue and Sobolev spaces over the half-space \(\mathbb{R}_+^n\), endowed with the norms \(\norm{\cdot}_p\) and \(\norm{\cdot}_{H^k}\), respectively.
The symbol \((\cdot,\cdot)\) is used for the standard \(L^2\) inner product.
We write \(A \lesssim B\) to indicate that the inequality \(A \leq C B\) holds for some generic positive constant \(C\).
The bracket notation \([q]\) refers to the floor function, i.e., the greatest integer less than or equal to \(q\).
We use \(\partial_x^k f\) with an integer \(k\geq 0\) to represent the collection of all \(k\)-th order partial derivatives of \(f\) with respect to the full spatial variable \(x\); similarly, \(\partial_{x'}^k f\) stands for the collection of all \(k\)-th order tangential derivatives with respect to \(x'\).
For later use, we set the regularity index \(s_0 \triangleq \left[\frac{n}{2}\right] + 1\).

Before stating our main results, we first introduce two fundamental functionals—the energy functional and the dissipation functional—whose explicit expressions are provided as follows.
\begin{align*}
   E_{\sigma}^{s}(t)\triangleq & \sup_{0\leq \tau \leq t}(1+\tau)^{\sigma}\norm{(a, \mathbf{v}, \mathbf{B})(\tau)}_{H^s} +\sum_{j=1}^{\left[\frac{s+1}{2}\right]}\sup_{0\leq \tau \leq t}(1+\tau)^{\sigma} \norm{\partial_{\tau}^{j} a(\tau)}_{H^{s+1-2 j}} \\
     &+ \sum_{j=1}^{\left[\frac{s}{2}\right]}\sup_{0\leq \tau \leq t} (1+\tau)^{\sigma}\norm{\partial_{\tau}^{j} (\mathbf{v}, \mathbf{B})(\tau)}_{H^{s-2 j}}
\end{align*}
and
\begin{align*}
D_{\sigma}^{s}(t)\triangleq \Bigg( &\int_{0}^{t}(1+\tau)^{2\sigma}\left(\norm{\partial_{x} a (\tau)}_{H^{s-1}}^{2}+\norm{\partial_{x}(\mathbf{v}, \mathbf{B})(\tau)}_{H^s}^{2}\right) \mathrm{d} \tau \\
&+\sum_{j=1}^{\left[\frac{s+1}{2}\right]} \int_{0}^{t}(1+\tau)^{2\sigma}\norm{\partial_{\tau}^{j} (a, \mathbf{v}, \mathbf{B})(\tau)}_{H^{s+1-2 j}}^{2} \mathrm{d} \tau \Bigg)^{\frac{1}{2}},
\end{align*}
where \(s \in \mathbb{Z}_+\) and $\sigma\geq 0$.

The global well-posedness of solutions to \eqref{1.4} is stated as follows.
\begin{theorem}[Global well-posedness and uniform estimate]\label{T1.1}
Let the spatial dimension \(n\) satisfy \(n \geq 2\). Suppose that for some integer regularity exponent \(s \geq s_0\), the initial data \((a_0,\mathbf{v}_0,\mathbf{B}_0) \in H^s(\mathbb{R}_+^n)\) fulfill the associated compatibility conditions. Then there exists a positive constant \(\eta_0 > 0\) such that if
\(
\norm{(a_0,\mathbf{v}_0,\mathbf{B}_0)}_{H^s} \leq \eta_0,
\)
then the initial-boundary value problem \eqref{1.4} admits a unique global solution \((a,\mathbf{v},\mathbf{B})\) in the class
\[
\begin{gathered}
a \in C\big([0, \infty) ; H^{s}\big), \quad \partial_{t}^{j} a \in C\big([0, \infty) ; H^{s+1-2 j}\big), \quad 1 \leq j \leq \left[\frac{s+1}{2}\right], \\
\partial_{t}^{j}(\mathbf{v}, \mathbf{B}) \in C\big([0, \infty) ; H^{s-2 j}\big), \quad 0 \leq j \leq \left[\frac{s}{2}\right], \\
\partial_{x} a \in L^{2}\big(0, \infty ; H^{s-1}\big), \quad \partial_{x}(\mathbf{v}, \mathbf{B}) \in L^{2}\big(0, \infty ; H^{s}\big), \\
\partial_{t}^{j} (a, \mathbf{v}, \mathbf{B}) \in L^{2}\big(0, \infty ; H^{s+1-2 j}\big), \quad 1 \leq j \leq \left[ \frac{s+1}{2}\right].
\end{gathered}
\]
Moreover, the uniform estimate
\begin{equation}\label{1.9}
  E_{0}^{s}(t)+D_{0}^{s}(t) \lesssim \eta_0
\end{equation}
holds for all \(t \geq 0\).
\end{theorem}

\begin{remark}
Owing to the condition \(s \geq s_0\), an application of the Sobolev embedding theorem allows us to choose \(\eta_0 > 0\) small enough in Theorem \ref{T1.1} so that the bound \(\norm{a(t)}_{L^{\infty}} \leq \frac{1}{2}\) holds for every \(t \geq 0\).
As a result, the density variable corresponding to system \eqref{1.1} stays uniformly bounded both from above and from below by positive constants over the whole time interval.
\end{remark}

Moving beyond global well-posedness, we next establish the decay rates of solutions to the nonlinear problem \eqref{1.1}–\eqref{1.3}.
This is achieved by combining the sharp decay estimates for solutions of the associated linearized compressible MHD system with Duhamel’s principle.
The complete statement of our asymptotic results is summarized in the theorem below.
\begin{theorem}[Asymptotic Behavior]\label{T1.3}
Let \(n \geq 2\) denote the number of spatial dimensions. Suppose the initial data \((a_0,\mathbf{m}_0,\mathbf{B}_0) \in H^s(\mathbb{R}_+^n) \cap L^1(\mathbb{R}_+^n)\) fulfill the associated compatibility conditions, where \(s \in \mathbb{Z}\) is an integer regularity exponent. Then there exists a positive constant \(\eta_1 > 0\) such that if
\(
    \norm{(a_0, \mathbf{m}_0, \mathbf{B}_0)}_{H^s} + \norm{(a_0, \mathbf{m}_0, \mathbf{B}_0)}_{L^1} \leq \eta_1,
\)
then the following statements hold:
\begin{enumerate}
\item[(i)] For \(s \geq s_0 + 1\), the perturbations \(\mathbf{u}(t) \triangleq (a, \mathbf{m})(t)\) and \(\mathbf{B}(t)\) satisfy
\[
\norm{(\mathbf{u},\mathbf{B})(t)}_{p} = O\left(t^{-\frac{n}{2}\left(1-\frac{1}{p}\right)}\right),
\]
\[
\norm{\partial_{x} \mathbf{u}(t)}_{2}=
\begin{cases}
O\left(t^{-\frac{n}{4}-\frac{3}{8}}\right) & n=2, \\
O\left(t^{-\frac{n}{4}-\frac{1}{2}}\right) & n \geq 3
\end{cases}
\]
and
\[
\norm{\partial_{x} \mathbf{B}(t)}_{2}= O\left(t^{-\frac{n}{4}-\frac{1}{2}}\right)
\]
for all \(2 \leq p \leq \infty\) as \(t \to \infty\).

\item[(ii)] For \(s \geq s_0 + 1\), it holds that
\[
\norm{\mathbf{u}(t)-\mathscr{U}(t) \mathbf{u}_{0}}_{2}=O\left(t^{-\frac{n}{4}-\frac{1}{4}} L(t)\right)
\]
as \(t \to \infty\), where \(L(t) \triangleq \log(t+1)\) for \(n=2\) and \(L(t) \triangleq 1\) for \(n \geq 3\).

\item[(iii)] For \(s \geq s_0 + 2\), it holds that
\[
\norm{\mathbf{u}(t)-\mathscr{U}(t) \mathbf{u}_{0}}_{\infty}=
\begin{cases}
O\left(t^{-\frac{n}{2}-\frac{3}{8}} \log(t+1)\right) & n=2, \\
O\left(t^{-\frac{n}{2}-\frac{1}{2}}\right) & n \geq 3
\end{cases}
\]
as \(t \to \infty\).

\item[(iv)] For \(s \geq s_0 + 1\), it holds that
\[
\norm{\mathbf{B}(t)-\mathscr{B}(t) \mathbf{B}_{0}}_{2}=O\left(t^{-\frac{n}{4}-\frac{1}{2}} L(t)\right)
\]
and
\[
\norm{\mathbf{B}(t)-\mathscr{B}(t) \mathbf{B}_{0}}_{\infty}=
\begin{cases}
O\left(t^{-\frac{n}{2}-\frac{3}{8}} \log(t+1)\right) & n=2, \\
O\left(t^{-\frac{n}{2}-\frac{1}{2}}\right) & n \geq 3
\end{cases}
\]
as \(t \to \infty\).
\end{enumerate}
\end{theorem}

\begin{remark}
Since the linear solution satisfies
\[
\norm{\left(\mathscr{U}(t) \mathbf{u}_{0},\mathscr{B}(t) \mathbf{B}_{0}\right)}_{p}=O\left(t^{-\frac{n}{2}\left(1-\frac{1}{p}\right)}\right)
\]
for \(2\leq p\leq\infty\) as \(t\to\infty\), the estimates in (ii), (iii) and (iv) show that the large-time leading-order profile of \((\mathbf{u},\mathbf{B})(t)\) is precisely given by the linear solution \(\left(\mathscr{U}(t) \mathbf{u}_{0},\mathscr{B}(t) \mathbf{B}_{0}\right)\).
\end{remark}
The main novelty of this paper lies in a unified analysis of global well-posedness and large-time asymptotic behavior of strong solutions to the \(n\)-dimensional (\(n\ge 2\)) compressible viscous resistive magnetohydrodynamics (MHD) equations in the upper half-space \(\mathbb{R}_+^n\). We impose the standard no-slip boundary condition on the velocity field, together with two physically meaningful magnetic boundary conditions: the insulating boundary condition and the perfectly conducting boundary condition.

Under the assumption of small initial perturbations around a constant equilibrium state, we first establish the global existence and uniqueness of strong solutions in Sobolev spaces. The proof builds on uniform energy estimates with a careful tangential-normal decomposition of boundary contributions, combined with a standard bootstrapping argument.

Second, leveraging the complete decoupling of the linearized MHD system into a linearized compressible Navier–Stokes subsystem and a vector-valued heat equation subsystem, we combine sharp decay estimates for the linearized solutions with Duhamel’s principle to derive optimal \(L^p\) decay rates of order \(t^{-\frac{n}{2}(1-\frac{1}{p})}\) for \(2\le p\le\infty\). We further obtain decay estimates for higher-order derivatives via a weighted energy method with time-dependent weight functions.

Furthermore, we prove that the global nonlinear solutions converge to the linearized solutions as \(t\to\infty\) with an improved convergence rate. This rigorously shows that the large-time leading-order behavior of the compressible MHD flow in the half-space is dominated by linear dynamics, whereas nonlinear interactions only produce higher-order corrections with faster decay.

This work complements the existing literature with a self-contained, unified treatment of global well-posedness and sharp large-time asymptotics for the compressible viscous resistive MHD system in the half-space, covering both insulating and perfectly conducting magnetic boundary conditions. In particular, it extends the classical results of Kagei and Kobayashi \cite{Kagei-Kobayashi2002, Kagei-Kobayashi2005} for half-space compressible Navier–Stokes equations to the MHD system under both magnetic boundary regimes.

The remainder of this paper is structured as follows. In Section \ref{sec:2}, we first establish the uniform energy estimates for strong solutions in the Sobolev space \(H^s\) with \(s\geq s_0\). Based on these uniform a priori estimates and a standard bootstrapping argument, we further complete the proof of the global well-posedness result in Theorem \ref{T1.1}. Section \ref{sec:3} is devoted to the large-time decay analysis of global solutions. By virtue of the sharp decay properties of solutions to the associated linearized MHD system together with Duhamel’s principle, we derive the  \(H^1\)  decay rates for the global solutions and rigorously establish the asymptotic behaviors summarized in Theorem \ref{T1.3}. Finally, Appendix \ref{sec:4} collects a series of fundamental inequalities and auxiliary lemmas, which serve as crucial technical tools throughout the whole proof.

\section{A priori estimates}\label{sec:2}
We devote the present section to deriving a priori estimates for the strong solutions \((a, \mathbf{v}, \mathbf{B})\) associated with the compressible MHD system \eqref{1.4}.
For all subsequent analysis in this section, we work under the assumptions \(s \geq s_0\) and that the solution obeys the a priori bound given by
\begin{equation}\label{2.1}
  E_0^s(t)+D_0^s(t) \leq \eta,
\end{equation}
with \(\eta > 0\) being a sufficiently small constant to be determined later.
\enlargethispage{1\baselineskip} 
\pagebreak[1] 
\subsection{Energy evolution}\label{sub:2.1}
This subsection is devoted to deriving several fundamental energy inequalities, which are collected in the proposition that follows.
As a starting point, we define the material derivative by
\[
\frac{D a}{D t}\triangleq\partial_{t} a+  \mathbf{v} \cdot \nabla a,
\]
and we also introduce the following auxiliary notations for subsequent use:
\[
A_{k, j}\triangleq\left(\partial_{x'}^{k} \partial_{t}^{j}f_1, \partial_{x'}^{k} \partial_{t}^{j} a\right)+\left(\partial_{x'}^{k} \partial_{t}^{j} {f}_2, \partial_{x'}^{k} \partial_{t}^{j} \mathbf{v}\right),
\]
\[
B_{k,j}\triangleq  \left( \partial _{x'}^{k}\partial _{t}^{j}f_1,\partial _{x'}^{k}\partial _{t}^{j+1}a \right) +\left( \partial _{x'}^{k}\partial _{t}^{j}{f}_2,\partial _{x'}^{k}\partial _{t}^{j+1}\mathbf{v}\right),
\]
\[
C_{k,\ell ,j}\triangleq  \left( \partial _{x'}^{k}\partial _{x_n}^{\ell +1}\partial _{t}^{j}(\mathbf{v}\cdot \nabla a ),\partial _{x'}^{k}\partial _{x_n}^{\ell +1}\partial _{t}^{j}a \right),
\]
\[
D_{k, j}\triangleq\left(\partial_{x'}^{k} \partial_{t}^{j} {f}_3, \partial_{x'}^{k} \partial_{t}^{j} \mathbf{B}\right),
\]
\[
E_{k, j}\triangleq\left(\partial_{x'}^{k} \partial_{t}^{j} {f}_3, \partial_{x'}^{k} \partial_{t}^{j+1} \mathbf{B}\right).
\]
\begin{proposition}\label{P2.1}
Let the regularity exponent \(s\) satisfy \(s \geq s_0\), and suppose \((a,\mathbf{v},\mathbf{B})\) is a solution of system \eqref{1.4}, we have the following inequalities, which hold uniformly for all \(t \geq 0\):

\begin{enumerate}
\item[(i)]
\begin{align*}
&(1+t)^{2\sigma} \norm{\partial_{x'}^{k} \partial_{t}^{j} (a,\mathbf{v})(t)}_{2}^{2}+\int_{0}^{t}(1+\tau)^{2\sigma}  \norm{\partial_{x'}^{k} \partial_{\tau}^{j} \left(\frac{D a}{D\tau},\partial_x\mathbf{v}\right)}_{2}^{2}  \mathrm{d} \tau \\
\lesssim {}&\sum_{j=0}^{\left[\frac{s}{2}\right]}\norm{\partial_{t}^{j} (a,\mathbf{v})(0)}_{H^{s-2j}}^{2}+\sigma\int_{0}^{t}  (1+\tau)^{2\sigma-1} \norm{\partial_{x'}^{k} \partial_{\tau}^{j} (a,\mathbf{v})}_{2}^{2} \mathrm{d} \tau\\
&+\int_{0}^{t}(1+\tau)^{2\sigma}  \left(\norm{\partial_{x'}^{k} \partial_{\tau}^{j}( a  \operatorname{div} \mathbf{v})}_{2}^{2}+\abs{A_{k,j}}\right)  \mathrm{d} \tau
\end{align*}
for $0\leq k+2j \leq s$,

\item[(ii)]
\begin{align*}
& (1+t)^{2\sigma}\norm{\partial_{x}\partial_{x'}^{k} \partial_{t}^{j} \mathbf{v}(t)}_{2}^{2}+\int_{0}^{t}  (1+\tau)^{2\sigma}\norm{\partial_{x'}^{k} \partial_{\tau}^{j+1} (a, \mathbf{v})}_{2}^{2}  \mathrm{d} \tau \\
\lesssim {}&\sum_{j=0}^{\left[\frac{s}{2}\right]}\norm{\partial_{t}^{j} (a,\mathbf{v})(0)}_{H^{s-2j}}^{2}+ (1+t)^{2\sigma}\norm{\partial_{x'}^{k} \partial_{t}^{j}  a (t)}_{2}^{2}\\
&+\sigma\int_{0}^{t}  (1+\tau)^{2\sigma-1} \left(\norm{\partial_{x'}^{k} \partial_{\tau}^{j}  a }_{2}^{2}+\norm{\partial_{x}\partial_{x'}^{k} \partial_{\tau}^{j} \mathbf{v}}_{2}^{2}\right) \mathrm{d} \tau\\
&+\int_{0}^{t} (1+\tau)^{2\sigma} \left(\norm{\partial_{x}\partial_{x'}^{k} \partial_{\tau}^{j} \mathbf{v}}_{2}^{2}+\abs{B_{k,j}}\right)  \mathrm{d} \tau
\end{align*}
for $ 0\leq k+2j \leq s-1$,

\item[(iii)]
\begin{align*}
&(1+t)^{2\sigma}\norm{\partial_{x'}^{k} \partial_{x_n}^{\ell+1} \partial_{t}^{j}  a (t)}_{2}^{2}+\int_{0}^{t}  (1+\tau)^{2\sigma}\norm{\partial_{x'}^{k} \partial_{x_n}^{\ell+1} \partial_{\tau}^{j} \left(a, \frac{D a }{D\tau}\right)}_{2}^{2}  \mathrm{d} \tau \\
\lesssim {}&\sum_{j=0}^{\left[\frac{s}{2}\right]}\norm{\partial_{t}^{j} (a,\mathbf{v})(0)}_{H^{s-2j}}^{2}
+\sigma\int_{0}^{t}  (1+\tau)^{2\sigma-1}\norm{\partial_{x'}^{k} \partial_{x_n}^{\ell+1} \partial_{\tau}^{j}  a }_{2}^{2}  \mathrm{d} \tau \\
&+\int_{0}^{t}  (1+\tau)^{2\sigma}\left(\norm{\partial_{x} \partial_{x'}^{k+1} \partial_{x_n}^{\ell} \partial_{\tau}^{j} \mathbf{v}}_{2}^{2}+\norm{\partial_{x'}^{k} \partial_{x_n}^{\ell} \partial_{\tau}^{j+1} \mathbf{v}}_{2}^{2}\right)  \mathrm{d} \tau \\
&+\int_{0}^{t}  (1+\tau)^{2\sigma}\left(\norm{\partial_{x'}^{k} \partial_{x_n}^{\ell+1} \partial_{\tau}^{j}( a  \operatorname{div} \mathbf{v})}_{2}^{2}+\norm{\partial_{x'}^{k} \partial_{x_n}^{\ell} \partial_{\tau}^{j}{f}_2}_{2}^{2}+\abs{C_{k,\ell,j}}\right)  \mathrm{d} \tau
\end{align*}
for $ 0\leq k+\ell+2j \leq s-1$,

\item[(iv)]
\begin{align*}
& (1+t)^{2\sigma}\norm{\partial_{x'}^{k} \partial_{x_n}^{\ell+2} \partial_{t}^{j} \mathbf{v}(t)}_{2}^{2} \\
\lesssim {}&  (1+t)^{2\sigma}\Bigg(\norm{\partial_{x} \partial_{x'}^{k} \partial_{x_n}^{\ell} \partial_{t}^{j}  a (t)}_{2}^{2}+\norm{\partial_{x} \partial_{x'}^{k+1} \partial_{x_n}^{\ell} \partial_{t}^{j} \mathbf{v}(t)}_{2}^{2}\\
&+\norm{\partial_{x'}^{k} \partial_{x_n}^{\ell} \partial_{t}^{j+1} \mathbf{v}(t)}_{2}^{2}+\norm{\partial_{x'}^{k} \partial_{x_n}^{\ell} \partial_{t}^{j} {f}_2(t)}_{2}^{2}\Bigg)
\end{align*}
for $ 0\leq k+\ell+2j \leq s-2$,

\item[(v)]
\begin{align*}
&\int_{0}^{t} (1+\tau)^{2\sigma} \left(\norm{\partial_{x}^{\ell+1} \partial_{x'}^{k} \partial_{\tau}^{j}  a }_{2}^{2}+\norm{\partial_{x}^{\ell+2} \partial_{x'}^{k} \partial_{\tau}^{j} \mathbf{v}}_{2}^{2}\right)  \mathrm{d} \tau \\
\lesssim {}& \int_{0}^{t}  (1+\tau)^{2\sigma}\left(\norm{\partial_{x}^{\ell+1} \partial_{x'}^{k} \partial_{\tau}^{j} \left(\frac{D a }{D\tau}\right)}_{2}^{2}+\norm{\partial_{x}^{\ell} \partial_{x'}^{k} \partial_{\tau}^{j+1} \mathbf{v}}_{2}^{2}\right. \\
&\left.+\norm{\partial_{x}^{\ell+1} \partial_{x'}^{k} \partial_{\tau}^{j}( a  \operatorname{div} \mathbf{v})}_{2}^{2}+\norm{\partial_{x}^{\ell} \partial_{x'}^{k} \partial_{\tau}^{j} {f}_2}_{2}^{2}\right)  \mathrm{d} \tau
\end{align*}
for $ 0\leq k+\ell+2j \leq s-1$,

\item[(vi)]
\begin{align*}
& (1+t)^{2\sigma}\norm{\partial_{x}^{k} \partial_{t}^{j+1}  a (t)}_{2}^{2}+\int_{0}^{t} (1+\tau)^{2\sigma} \norm{\partial_{x}^{k} \partial_{\tau}^{j+1}  a }_{2}^{2}  \mathrm{d} \tau \\
\lesssim {}& (1+t)^{2\sigma}\norm{\partial_{x}^{k+1} \partial_{t}^{j} \mathbf{v}(t)}_{2}^{2}+\int_{0}^{t}(1+\tau)^{2\sigma}  \norm{\partial_{x}^{k+1} \partial_{\tau}^{j} \mathbf{v}}_{2}^{2}  \mathrm{d} \tau \\
&+ (1+t)^{2\sigma}\Bigg(\norm{\partial_{x}^{k} \partial_{t}^{j}( a  \operatorname{div} \mathbf{v})(t)}_{2}^{2}+\norm{\partial_{x}^{k} \partial_{t}^{j}(\mathbf{v} \cdot \nabla  a )(t)}_{2}^{2}\Bigg) \\
&+ \int_{0}^{t} (1+\tau)^{2\sigma} \left(\norm{\partial_{x}^{k} \partial_{\tau}^{j}( a  \operatorname{div} \mathbf{v})}_{2}^{2}+\norm{\partial_{x}^{k} \partial_{\tau}^{j}(\mathbf{v} \cdot \nabla  a )}_{2}^{2}\right)  \mathrm{d} \tau
\end{align*}
for $ 0\leq k+2j \leq s-1$,

\item[(vii)]
\begin{align*}
& (1+t)^{2\sigma}\norm{\partial_{x'}^{k} \partial_{t}^{j} \mathbf{B}(t)}_{2}^{2}+\int_{0}^{t} (1+\tau)^{2\sigma} \norm{ \partial_{x}\partial_{x'}^{k} \partial_{\tau}^{j} \mathbf{B}}_{2}^{2}  \mathrm{d} \tau \\
\lesssim {}&\sum_{j=0}^{\left[\frac{s}{2}\right]}\norm{\partial_{t}^{j} \mathbf{B}(0)}_{H^{s-2j}}^{2}+\sigma\int_{0}^{t}  (1+\tau)^{2\sigma-1} \norm{\partial_{x'}^{k} \partial_{\tau}^{j} \mathbf{B}}_{2}^{2} \mathrm{d} \tau+\int_{0}^{t}(1+\tau)^{2\sigma} \abs{D_{k,j}}  \mathrm{d} \tau
\end{align*}
for $ 0\leq k+2j \leq s$,

\item[(viii)]
\begin{align*}
&(1+t)^{2\sigma} \norm{\partial_{x}\partial_{x'}^{k} \partial_{t}^{j} \mathbf{B}(t)}_{2}^{2}+\int_{0}^{t}(1+\tau)^{2\sigma}  \norm{\partial_{x'}^{k} \partial_{\tau}^{j+1} \mathbf{B}}_{2}^{2}  \mathrm{d} \tau \\
\lesssim {}&\sum_{j=0}^{\left[\frac{s}{2}\right]}\norm{\partial_{t}^{j} \mathbf{B}(0)}_{H^{s-2j}}^{2}+\sigma\int_{0}^{t}  (1+\tau)^{2\sigma-1} \norm{\partial_{x}\partial_{x'}^{k} \partial_{\tau}^{j} \mathbf{B}}_{2}^{2} \mathrm{d} \tau+\int_{0}^{t} (1+\tau)^{2\sigma}\abs{E_{k,j}}  \mathrm{d} \tau
\end{align*}
for $ 0\leq k+2j \leq s-1$,

\item[(ix)]
\begin{align*}
& (1+t)^{2\sigma}\norm{\partial_{x'}^{k} \partial_{x_n}^{\ell+2} \partial_{t}^{j} \mathbf{B}(t)}_{2}^{2} \\
\lesssim {}&  (1+t)^{2\sigma}\Bigg(\norm{\partial_{x'}^{k+2} \partial_{x_n}^{\ell} \partial_{t}^{j} \mathbf{B}(t)}_{2}^{2}+\norm{\partial_{x'}^{k} \partial_{x_n}^{\ell} \partial_{t}^{j+1} \mathbf{B}(t)}_{2}^{2}+\norm{\partial_{x'}^{k} \partial_{x_n}^{\ell} \partial_{t}^{j} {f}_3(t)}_{2}^{2}\Bigg)
\end{align*}
for $ 0\leq k+\ell+2j \leq s-2$,

\item[(x)]
\begin{align*}
& \int_{0}^{t}(1+\tau)^{2\sigma}\norm{\partial_{x'}^{k} \partial_{x_n}^{\ell+2} \partial_{t}^{j} \mathbf{B}}_{2}^{2}  \mathrm{d} \tau \\
\lesssim {}&  \int_{0}^{t}(1+\tau)^{2\sigma}\Bigg(\norm{\partial_{x'}^{k+2} \partial_{x_n}^{\ell} \partial_{\tau}^{j} \mathbf{B}}_{2}^{2}+\norm{\partial_{x'}^{k} \partial_{x_n}^{\ell} \partial_{\tau}^{j+1} \mathbf{B}}_{2}^{2}+\norm{\partial_{x'}^{k} \partial_{x_n}^{\ell} \partial_{\tau}^{j} {f}_3}_{2}^{2}\Bigg) \mathrm{d} \tau
\end{align*}
for $ 0\leq k+\ell+2j \leq s-1$.
\end{enumerate}
\end{proposition}

\begin{proof}
The argument for statements (i) through (vi) is identical to the one presented in Proposition 11.3 of \cite{Kagei-Kobayashi2005}.
For items (vii) and (viii), one may proceed analogously to the proofs of (i) and (ii), while the derivations of (ix) and (x) closely parallel that of estimate (iv).
In light of this, we do not reproduce the full details of the proof.
\end{proof}
\subsection{Synthesis}\label{sub:2.2}
All the estimates derived in Proposition \ref{P2.1} are now combined to yield the proposition that follows.
To facilitate the statement of this result, we define the following nonlinearity for notational simplicity:
\begin{align*}
N_{\sigma}^{s}(t)\triangleq&\left(\sum_{j=0}^{\left[\frac{s-1}{2}\right]}(1+t)^{2\sigma}\left( \norm{\partial_{t}^{j}\left( a  \operatorname{div} \mathbf{v}\right)}_{H^{s-1-2j}}^{2}+\norm{\partial_{t}^{j} \left(\mathbf{v} \cdot \nabla  a\right)}_{H^{s-1-2j}}^{2}\right)\right.\\
& +\sum_{j=0}^{\left[\frac{s-2}{2}\right]}(1+t)^{2\sigma} \norm{\partial_{t}^{j}({f}_2,{f}_3)}_{H^{s-2-2j}}^{2}+\sum_{j=0}^{\left[\frac{s}{2}\right]} \int_{0}^{t} (1+\tau)^{2\sigma}\norm{\partial_{\tau}^{j}\left( a  \operatorname{div} \mathbf{v}\right)}_{H^{s-2j}}^{2} \mathrm{d} \tau\\
& +\sum_{j=0}^{\left[\frac{s-1}{2}\right]}\int_{0}^{t} (1+\tau)^{2\sigma}\left(\norm{\partial_{\tau}^{j} \left(\mathbf{v} \cdot \nabla  a\right)}_{H^{s-1-2j}}^{2}+\norm{\partial_{\tau}^{j} \left({f}_2,{f}_3\right)}_{H^{s-1-2j}}^{2}\right) \mathrm{d} \tau \\
& +\sum_{0 \leq k+2 j \leq s} \int_{0}^{t}  (1+\tau)^{2\sigma}\left(  \abs{A_{k,j}} + \abs{D_{k,j}}  \right)  \mathrm{d} \tau+\sum_{0 \leq k+\ell+2 j \leq s-1} \int_{0}^{t} (1+\tau)^{2\sigma}\abs{C_{k,\ell,j}} \mathrm{d} \tau\\
& \left. + \sum_{0 \leq k+2 j \leq s-1} \int_{0}^{t}(1+\tau)^{2\sigma}\left(  \abs{B_{k,j}} + \abs{E_{k,j}}  \right)\mathrm{d} \tau\right)^{\frac{1}{2}},
\end{align*}
then we can obtain the following
\begin{proposition}\label{P2.2}
Let the regularity exponent \(s\) satisfy \(s \geq s_0\), and suppose \((a,\mathbf{v},\mathbf{B})\) is a solution of system \eqref{1.4}. Then the inequality below holds for all \(t \geq 0\):
\begin{equation}\label{2.2}
\begin{aligned}
E_{\sigma}^{s}(t)^{2}+D_{\sigma}^{s}(t)^{2} \lesssim{}& \norm{( a _0,\mathbf{v}_0,\mathbf{B}_0)}_{H^s}^2+N_{\sigma}^{s}(t)^{2}+\sigma\bigg(\delta D_{\sigma}^{s}(t)^{2}+D_{0}^{s}(t)^{2}\\
&+\int_{0}^{t}  (1+\tau)^{2\sigma-1} \norm{({a}, \mathbf{v},\mathbf{B})}_{2}^{2} \mathrm{d} \tau\bigg)
\end{aligned}
\end{equation}
for any \(\delta>0\).
\end{proposition}
\begin{proof}
As a preliminary step, we recall the following elementary weighted inequality:
\[
\int_{0}^{t}  (1+\tau)^{2\sigma-1} \abs{f}^{2} \mathrm{d} \tau\lesssim \delta\int_{0}^{t}  (1+\tau)^{2\sigma}\abs{f}^{2} \mathrm{d} \tau+\int_{0}^{t}\abs{f}^{2} \mathrm{d} \tau
\]
for arbitrary \(\delta>0\).

Applying statement (i) of Proposition \ref{P2.1} then gives
\begin{equation}\label{2.3}
  \begin{aligned}
& \sum_{0 \leq k+2 j \leq s}\left((1+t)^{2\sigma}\norm{\partial_{x'}^{k} \partial_{t}^{j} (a,\mathbf{v})(t)}_{2}^{2}+\int_{0}^{t}  (1+\tau)^{2\sigma}\norm{\partial_{x'}^{k} \partial_{\tau}^{j} \left(\frac{D a}{D\tau},\partial_x\mathbf{v}\right)}_{2}^{2}  \mathrm{d} \tau \right)\\
\lesssim {}& \norm{( a _0,\mathbf{v}_0,\mathbf{B}_0)}_{H^s}^2+N_{\sigma}^{s}(t)^{2}+\sigma\bigg(\delta D_{\sigma}^{s}(t)^{2}+D_{0}^{s}(t)^{2}+\int_{0}^{t}  (1+\tau)^{2\sigma-1} \norm{({a}, \mathbf{v},\mathbf{B})}_{2}^{2} \mathrm{d} \tau\bigg).\\
\end{aligned}
\end{equation}

Substituting estimate \eqref{2.3} into statement (ii) of Proposition \ref{P2.1}, we deduce that
\begin{equation}\label{2.4}
  \begin{aligned}
& \sum_{0 \leq k+2j \leq s-1}\left((1+t)^{2\sigma}\norm{\partial_{x}\partial_{x'}^{k} \partial_{t}^{j} \mathbf{v}(t)}_{2}^{2}+\int_{0}^{t} (1+\tau)^{2\sigma} \norm{\partial_{x'}^{k} \partial_{\tau}^{j+1} (a, \mathbf{v})}_{2}^{2}  \mathrm{d} \tau\right) \\
\lesssim {}& \norm{( a _0,\mathbf{v}_0,\mathbf{B}_0)}_{H^s}^2+N_{\sigma}^{s}(t)^{2}+\sigma\bigg(\delta D_{\sigma}^{s}(t)^{2}+D_{0}^{s}(t)^{2}+\int_{0}^{t}  (1+\tau)^{2\sigma-1} \norm{({a}, \mathbf{v},\mathbf{B})}_{2}^{2} \mathrm{d} \tau\bigg).\\
\end{aligned}
\end{equation}

We proceed to establish the families of estimates \(\eqref{2.5}_{\ell}\), \(\eqref{2.7}_{\ell}\) for \(0 \leq \ell \leq s-1\) and \(\eqref{2.6}_{\ell}\) for \(0 \leq \ell \leq s-2\). The proof proceeds by induction on the index \(\ell\):
\begin{equation}\label{2.5}
  \begin{aligned}
& \sum_{0 \leq k+\ell+2j \leq s-1}\left((1+t)^{2\sigma}\norm{\partial_{x'}^{k} \partial_{x_n}^{\ell+1} \partial_{t}^{j}  a (t)}_{2}^{2}+\int_{0}^{t} (1+\tau)^{2\sigma} \norm{\partial_{x'}^{k} \partial_{x_n}^{\ell+1} \partial_{\tau}^{j} \left(a, \frac{D a }{D\tau}\right)}_{2}^{2}  \mathrm{d} \tau\right) \\
\lesssim {}& \norm{( a _0,\mathbf{v}_0,\mathbf{B}_0)}_{H^s}^2+N_{\sigma}^{s}(t)^{2}+\sigma\bigg(\delta D_{\sigma}^{s}(t)^{2}+D_{0}^{s}(t)^{2}+\int_{0}^{t}  (1+\tau)^{2\sigma-1} \norm{({a}, \mathbf{v},\mathbf{B})}_{2}^{2} \mathrm{d} \tau\bigg),\\
\end{aligned}
\end{equation}
\begin{equation}\label{2.6}
\begin{aligned}
 &\sum_{0 \leq k+\ell+2j \leq s-2}(1+t)^{2\sigma}\norm{\partial_{x'}^{k} \partial_{x_n}^{\ell+2} \partial_{t}^{j} \mathbf{v}(t)}_{2}^{2}\\
\lesssim {}& \norm{( a _0,\mathbf{v}_0,\mathbf{B}_0)}_{H^s}^2+N_{\sigma}^{s}(t)^{2}+\sigma\bigg(\delta D_{\sigma}^{s}(t)^{2}+D_{0}^{s}(t)^{2}+\int_{0}^{t}  (1+\tau)^{2\sigma-1} \norm{({a}, \mathbf{v},\mathbf{B})}_{2}^{2} \mathrm{d} \tau\bigg),\\
\end{aligned}
\end{equation}
\begin{equation}\label{2.7}
  \begin{aligned}
&\sum_{0 \leq k+\ell+2j \leq s-1}\int_{0}^{t}(1+\tau)^{2\sigma}\left(\norm{\partial_{x}^{\ell+1} \partial_{x'}^{k} \partial_{\tau}^{j}  a }_{2}^{2}+\norm{\partial_{x}^{\ell+2} \partial_{x'}^{k} \partial_{\tau}^{j} \mathbf{v}}_{2}^{2}\right)  \mathrm{d} \tau\\
\lesssim {}& \norm{( a _0,\mathbf{v}_0,\mathbf{B}_0)}_{H^s}^2+N_{\sigma}^{s}(t)^{2}+\sigma\bigg(\delta D_{\sigma}^{s}(t)^{2}+D_{0}^{s}(t)^{2}+\int_{0}^{t}  (1+\tau)^{2\sigma-1} \norm{({a}, \mathbf{v},\mathbf{B})}_{2}^{2} \mathrm{d} \tau\bigg).\\
\end{aligned}
\end{equation}

We begin with the base case \(\ell = 0\). Estimate \(\eqref{2.5}_{0}\) is a direct consequence of statement (iii) in Proposition \ref{P2.1} with \(\ell=0\), combined with inequalities \eqref{2.3} and \eqref{2.4}. Likewise, estimate \(\eqref{2.6}_{0}\) follows from statement (iv) of Proposition \ref{P2.1} with \(\ell=0\), together with \eqref{2.3}, \eqref{2.4} and \(\eqref{2.5}_{0}\). By a similar argument, \(\eqref{2.7}_{0}\) is derived from statement (v) of Proposition \ref{P2.1} with \(\ell=0\), \eqref{2.3}, \eqref{2.4} and \(\eqref{2.5}_{0}\). Hence \(\eqref{2.5}_{\ell}\)–\(\eqref{2.7}_{\ell}\) all hold for \(\ell=0\); in particular, this covers the full range \(0 \leq \ell \leq s-2\) when \(s=2\).

We next turn to the case \(\ell = 1\). By invoking statement (iii) of Proposition \ref{P2.1} at level \(\ell=1\), together with \eqref{2.3} and \(\eqref{2.7}_{0}\), one obtains estimate \(\eqref{2.5}_{1}\). In turn, \(\eqref{2.6}_{1}\) follows from statement (iv) of Proposition \ref{P2.1} with \(\ell=1\), using \eqref{2.4}, \(\eqref{2.5}_{0}\), \(\eqref{2.5}_{1}\) and \(\eqref{2.6}_{0}\). Similarly, \(\eqref{2.7}_{1}\) is deduced from statement (v) of Proposition \ref{P2.1} at \(\ell=1\), combined with \eqref{2.3}, \(\eqref{2.5}_{0}\) and \(\eqref{2.5}_{1}\). This proves all three estimates for \(\ell=1\), which covers the full range when \(s=3\).

For general \(s \geq 4\) and any index \(2 \leq \ell \leq s-2\), we assume as induction hypothesis that \(\eqref{2.5}_{m}\)–\(\eqref{2.7}_{m}\) are valid for all \(0 \leq m \leq \ell-1\). Under this hypothesis, estimate \(\eqref{2.5}_{\ell}\) follows from statement \((iii)_{\ell}\) of Proposition \ref{P2.1} together with \(\eqref{2.7}_{m}\) for \(m = \ell-2, \ell-1\). In turn, estimate \(\eqref{2.6}_{\ell}\) is obtained from statement \((iv)_{\ell}\) of Proposition \ref{P2.1}, using \(\eqref{2.5}_{m}\) with \(m = \ell-1, \ell\) and \(\eqref{2.6}_{m}\) with \(m = \ell-2, \ell-1\). Finally, estimate \(\eqref{2.7}_{\ell}\) follows from statement \((v)_{\ell}\) of Proposition \ref{P2.1}, combined with \eqref{2.3}, the family \(\eqref{2.5}_{m}\) for \(0 \leq m \leq \ell\), and \(\eqref{2.7}_{\ell-2}\). This completes the inductive step, so that \(\eqref{2.5}_{\ell}\), \(\eqref{2.6}_{\ell}\) and \(\eqref{2.7}_{\ell}\) hold for all \(0 \leq \ell \leq s-2\) whenever \(s \geq 2\).

By a similar reasoning, the estimate \(\eqref{2.5}_{s-1}\) corresponding to \(\ell = s-1\) is derived from statement (iii) of Proposition \ref{P2.1} with \(\ell = s-1\), combined with \eqref{2.3}, \(\eqref{2.7}_{s-3}\) and \(\eqref{2.7}_{s-2}\). The estimate \(\eqref{2.7}_{s-1}\) then follows from statement (v) of Proposition \ref{P2.1} at \(\ell = s-1\), together with \eqref{2.3}, the full family \(\eqref{2.5}_{\ell}\) for \(0 \leq \ell \leq s-1\), and \(\eqref{2.7}_{s-3}\). This concludes the proof of estimates \(\eqref{2.5}_{\ell}\) and \(\eqref{2.7}_{\ell}\) for \(0 \leq \ell \leq s-1\), as well as \(\eqref{2.6}_{\ell}\) for \(0 \leq \ell \leq s-2\).

Plugging estimates \eqref{2.3}, \eqref{2.4}, \eqref{2.6} and \eqref{2.7} into statement (vi) of Proposition \ref{P2.1} yields
\begin{equation}\label{2.8}
 \begin{aligned}
& \sum_{0 \leq k+2j \leq s-1}\left((1+t)^{2\sigma}\norm{\partial_{x}^{k} \partial_{t}^{j+1}  a (t)}_{2}^{2}+\int_{0}^{t} (1+\tau)^{2\sigma} \norm{\partial_{x}^{k} \partial_{\tau}^{j+1}  a }_{2}^{2}  \mathrm{d} \tau\right) \\
\lesssim {}& \norm{( a _0,\mathbf{v}_0,\mathbf{B}_0)}_{H^s}^2+N_{\sigma}^{s}(t)^{2}+\sigma\bigg(\delta D_{\sigma}^{s}(t)^{2}+D_{0}^{s}(t)^{2}+\int_{0}^{t}  (1+\tau)^{2\sigma-1} \norm{({a}, \mathbf{v},\mathbf{B})}_{2}^{2} \mathrm{d} \tau\bigg).\\
\end{aligned}
\end{equation}

Turning to the magnetic field, we deduce from statements (vii) and (viii) of Proposition \ref{P2.1} that
\begin{equation}\label{2.9}
  \begin{aligned}
& \sum_{0 \leq k+2 j \leq s}\left((1+t)^{2\sigma}\norm{\partial_{x'}^{k} \partial_{t}^{j} \mathbf{B}(t)}_{2}^{2}+\int_{0}^{t} (1+\tau)^{2\sigma} \norm{\partial_{x}\partial_{x'}^{k} \partial_{\tau}^{j}\mathbf{B}}_{2}^{2}  \mathrm{d} \tau\right) \\
\lesssim{}& \norm{( a _0,\mathbf{v}_0,\mathbf{B}_0)}_{H^s}^2+N_{\sigma}^{s}(t)^{2}+\sigma\bigg(\delta D_{\sigma}^{s}(t)^{2}+D_{0}^{s}(t)^{2}+\int_{0}^{t}  (1+\tau)^{2\sigma-1} \norm{({a}, \mathbf{v},\mathbf{B})}_{2}^{2} \mathrm{d} \tau\bigg)\\
\end{aligned}
\end{equation}
and
\begin{equation}\label{2.10}
  \begin{aligned}
& \sum_{0 \leq k+2 j \leq s-1}\left((1+t)^{2\sigma} \norm{\partial_{x}\partial_{x'}^{k} \partial_{t}^{j} \mathbf{B}(t)}_{2}^{2}+\int_{0}^{t}(1+\tau)^{2\sigma}  \norm{\partial_{x'}^{k} \partial_{\tau}^{j+1} \mathbf{B}}_{2}^{2}  \mathrm{d} \tau \right) \\
\lesssim{}& \norm{( a _0,\mathbf{v}_0,\mathbf{B}_0)}_{H^s}^2+N_{\sigma}^{s}(t)^{2}+\sigma\bigg(\delta D_{\sigma}^{s}(t)^{2}+D_{0}^{s}(t)^{2}+\int_{0}^{t}  (1+\tau)^{2\sigma-1} \norm{({a}, \mathbf{v},\mathbf{B})}_{2}^{2} \mathrm{d} \tau\bigg).\\
\end{aligned}
\end{equation}

We next establish the families \(\eqref{2.11}_{\ell}\) for \(0 \leq \ell \leq s-2\) and \(\eqref{2.12}_{\ell}\) for \(0 \leq \ell \leq s-1\), again by induction on \(\ell\):
\begin{equation}\label{2.11}
\begin{aligned}
 &\sum_{0 \leq k+\ell+2j \leq s-2}(1+t)^{2\sigma}\norm{\partial_{x'}^{k} \partial_{x_n}^{\ell+2} \partial_{t}^{j} \mathbf{B}(t)}_{2}^{2}\\
\lesssim{}& \norm{( a _0,\mathbf{v}_0,\mathbf{B}_0)}_{H^s}^2+N_{\sigma}^{s}(t)^{2}+\sigma\bigg(\delta D_{\sigma}^{s}(t)^{2}+D_{0}^{s}(t)^{2}+\int_{0}^{t}  (1+\tau)^{2\sigma-1} \norm{({a}, \mathbf{v},\mathbf{B})}_{2}^{2} \mathrm{d} \tau\bigg),\\
\end{aligned}
\end{equation}
\begin{equation}\label{2.12}
\begin{aligned}
&\sum_{0 \leq k+\ell+2j \leq s-1}\int_{0}^{t}(1+\tau)^{2\sigma}\norm{\partial_{x'}^{k} \partial_{x_n}^{\ell+2} \partial_{t}^{j} \mathbf{B}}_{2}^{2}\mathrm{d} \tau\\
\lesssim{}& \norm{( a _0,\mathbf{v}_0,\mathbf{B}_0)}_{H^s}^2+N_{\sigma}^{s}(t)^{2}+\sigma\bigg(\delta D_{\sigma}^{s}(t)^{2}+D_{0}^{s}(t)^{2}+\int_{0}^{t}  (1+\tau)^{2\sigma-1} \norm{({a}, \mathbf{v},\mathbf{B})}_{2}^{2} \mathrm{d} \tau\bigg).\\
\end{aligned}
\end{equation}

For the base case \(\ell = 0\), estimate \(\eqref{2.11}_{0}\) follows directly from statement (ix) of Proposition \ref{P2.1} with \(\ell=0\) and inequality \eqref{2.9}. Similarly, \(\eqref{2.12}_{0}\) is obtained from statement (x) of Proposition \ref{P2.1} with \(\ell=0\), combined with \eqref{2.9} and \eqref{2.10}. This proves both estimates at level \(\ell=0\); in particular, for \(s=2\), the full range \(0 \leq \ell \leq s-2\) is already covered.

Moving to \(\ell = 1\), estimate \(\eqref{2.11}_{1}\) follows from statement (ix) of Proposition \ref{P2.1} with \(\ell=1\) and inequality \eqref{2.10}. By the same token, \(\eqref{2.12}_{1}\) is derived from statement (x) of Proposition \ref{P2.1} with \(\ell=1\) together with \eqref{2.9}. Hence the estimates hold for \(\ell=1\), which covers the full range for \(s=3\).

For \(s \geq 4\) and \(2 \leq \ell \leq s-2\), assume inductively that \(\eqref{2.11}_{m}\) and \(\eqref{2.12}_{m}\) hold for all \(0 \leq m \leq \ell-1\). Then statement \((ix)_{\ell}\) of Proposition \ref{P2.1} combined with \(\eqref{2.11}_{\ell-2}\) yields \(\eqref{2.11}_{\ell}\), and statement \((x)_{\ell}\) together with \(\eqref{2.12}_{\ell-2}\) yields \(\eqref{2.12}_{\ell}\). By induction, both families of estimates are valid for all \(0 \leq \ell \leq s-2\) when \(s \geq 2\).

Finally, the estimate \(\eqref{2.12}_{s-1}\) at the highest level \(\ell = s-1\) follows from statement (x) of Proposition \ref{P2.1} with \(\ell = s-1\), together with \eqref{2.9} and \(\eqref{2.12}_{s-3}\). This completes the proof of \(\eqref{2.11}_{\ell}\) for \(0 \leq \ell \leq s-2\) and \(\eqref{2.12}_{\ell}\) for \(0 \leq \ell \leq s-1\).

Summing all the foregoing estimates \eqref{2.3}-\eqref{2.12} yields the desired inequality \eqref{2.2}, which concludes the proof.
\end{proof}

\subsection{Estimate of the Nonlinearity \(N_{\sigma}^{s}(t)\)}\label{sub:2.3}
To complete the energy estimate, it remains to derive the bound for the nonlinear term \(N_{\sigma}^{s}(t)\), which is provided in the proposition below.

\begin{proposition}\label{P2.3}
Let the regularity index \(s\) satisfy \(s \geq s_0\), and suppose that the energy satisfies \(E_{0}^{s}(t) \leq 1\) for every \(t \geq 0\). Then the inequality below is valid for all \(t \geq 0\):
  \begin{equation}\label{2.13}
    N_{\sigma}^{s}(t)^2\lesssim \delta \left(E_{\sigma}^{s}(t)^2+D_{\sigma}^{s}(t)^2\right)+\left(E_{\sigma}^{s}(t)^4+D_{\sigma}^{s}(t)^4\right)
  \end{equation}
for any \(\delta>0\).
\end{proposition}

\begin{proof}
The desired bound follows by applying Lemmas \ref{L4.1}-\ref{L4.3}, following an argument analogous to that in \cite{Kagei-Kobayashi2005}; we therefore omit the full details of the derivation.
\end{proof}
\subsection{Global energy estimate}\label{sub:2.4}
Substituting estimate \eqref{2.13} into inequality \eqref{2.2} and choosing \(\delta\) sufficiently small, we deduce that

\begin{proposition}\label{P2.4}
Let the regularity index \(s\) satisfy \(s \geq s_0\), and suppose that \(E_{0}^{s}(t) \leq 1\) for every \(t \geq 0\).
Then the inequality below holds for all \(t \geq 0\):
\begin{equation}\label{2.14}
  \begin{aligned}
E_{\sigma}^{s}(t)^{2}+D_{\sigma}^{s}(t)^{2} \lesssim{}& \norm{(a_0, \mathbf{v}_{0}, \mathbf{B}_0)}_{H^s}^{2}+E_{\sigma}^{s}(t)^{4}+D_{\sigma}^{s}(t)^{4}\\
&+\sigma\bigg(D_{0}^{s}(t)^{2}+\int_{0}^{t}  (1+\tau)^{2\sigma-1} \norm{({a}, \mathbf{v},\mathbf{B})}_{2}^{2} \mathrm{d} \tau\bigg).\\
\end{aligned}
\end{equation}
\end{proposition}

Setting \(\sigma = 0\), the energy estimate \eqref{1.9} from Theorem \ref{T1.1} follows directly from the a priori bound \eqref{2.1} combined with inequality \eqref{2.14}.

\section{Asymptotic behavior of global solutions}\label{sec:3}
In this section we investigate large-time behavior of solutions to the nonlinear problem \eqref{1.5}.
\subsection{Decay rates of linear operators}\label{sub:3.1}
Solutions to system \eqref{1.5} admit an integral representation via Duhamel’s principle.
Building on the decay rate estimates for solutions of the associated linearized system, we derive the decay rates of solutions to the nonlinear problem by means of weighted energy estimates.
To characterize the decay behavior of the linearized operator, we split the analysis into two parts. The first part concerns the decay rates of solutions to the linearized compressible Navier–Stokes equations on the half-space under the no-slip boundary condition.

The decay rates for the semigroup \(\mathscr{U}(t) \mathbf{u}_{0}\) were already established in \cite{Kagei-Kobayashi2005}. We simply recall the relevant results needed in the present paper in Propositions \ref{P3.1}–\ref{P3.3} and omit the corresponding proofs.

For notational simplicity, we introduce the function \(q(\ell)\) for \(\ell \geq 0\), defined by
\begin{equation}\label{3.1}
  q(\ell) \triangleq
\begin{cases}
\min\left\{ \dfrac{\ell}{2}, \dfrac{\ell}{4} + \dfrac{1}{8} \right\} & (n = 2), \\[1.2em]
\min\left\{ \dfrac{\ell}{2}, \dfrac{\ell}{4} + \dfrac{3}{8} \right\} & (n \ge 3).
\end{cases}
\end{equation}
Analogously, we define the regularity index \(k_p(\ell)\) for \(\ell \in \mathbb{Z}\) and \(2 \leq p \leq \infty\) by
\begin{equation}\label{3.2}
  k_p(\ell) \triangleq
\begin{cases}
\max\left\{ 0, \ell \right\} & (p = 2), \\[1.2em]
\max\left\{ 0, \left[\ell + n\left( \dfrac{1}{2} - \dfrac{1}{p} \right)\right] \right\} + 1 & (2 < p \le \infty).
\end{cases}
\end{equation}

Given a vector \(\mathbf{u} = (a, \mathbf{m})\) with \(a \in W^{\ell,p}\) and \(\mathbf{m} = (m_1, \dots, m_n) \in W^{k,q}\), the norm on the product space \(W^{\ell,p} \times W^{k,q}\), denoted \(\norm{\mathbf{u}}_{W^{\ell,p} \times W^{k,q}}\), is given by
\[
\norm{\mathbf{u}}_{W^{\ell,p} \times W^{k,q}} \triangleq \norm{a}_{W^{\ell,p}} + \norm{\mathbf{m}}_{W^{k,q}}.
\]

\begin{proposition}\label{P3.1}
  Let \(n \geq 2\) be the spatial dimension, let \(\ell \in \mathbb{Z}\) and let \(2 \leq p \leq \infty\). Then there exists a positive constant \(c\) such that the estimates below hold uniformly for all \(t \geq 1\).
  \begin{enumerate}
    \item[(i)] For \(\ell \leq n-1\), one has
    \begin{align*}
  \norm{\partial_x^\ell \mathscr{U}(t) \mathbf{u}_0}_p \lesssim{}
& t^{-\frac{n}{2}\left(1-\frac{1}{p}\right)-\frac{2}{p}q\left(\ell\right)-\left(1-\frac{2}{p}\right)\frac{\ell}{2}} \norm{\mathbf{u}_0}_1 \\
& + e^{-c t} \left(\norm{a_0}_{H^{k_p\left(\ell\right)}}+ \norm{\mathbf{m}_0}_{H^{k_p\left(\ell-1\right)}}  \right).
\end{align*}

    \item[(ii)] For \(\ell \geq n\), it holds that
    \begin{align*}
      &\norm{\partial_x^\ell \mathscr{U}(t) \mathbf{u}_0}_p\\
     \lesssim&
      \begin{cases}
            t^{-\frac{n}{2}\left(1-\frac{1}{p}\right) - \frac{n-1}{4}\left(1-\frac{2}{p}\right) - \frac{1}{4p} - \frac{\ell}{4}} \norm{\mathbf{u}_0}_1
            + e^{-c t} \left(\norm{a_0}_{H^{k_p\left(\ell\right)}}+ \norm{\mathbf{m}_0}_{H^{k_p\left(\ell-1\right)}}  \right), & n=2, \\[1.5em]
            t^{-\frac{n}{2}\left(1-\frac{1}{p}\right) - \frac{n-1}{4}\left(1-\frac{2}{p}\right) - \frac{3}{4p} - \frac{\ell}{4}} \norm{\mathbf{u}_0}_1
            + e^{-c t} \left(\norm{a_0}_{H^{k_p\left(\ell\right)}}+ \norm{\mathbf{m}_0}_{H^{k_p\left(\ell-1\right)}}  \right), & n \ge 3.
    \end{cases}
    \end{align*}
  \end{enumerate}
\end{proposition}

\begin{proposition}\label{P3.2}
  Let \(\mathbf{u}_0 = (0, \mathbf{m}_0)\), and fix an exponent \(2 \leq p \leq \infty\). Then there exists a positive constant \(c > 0\) such that the estimate
\[
\norm{\mathscr{U}(t) \partial_x \mathbf{u}_0}_p \lesssim t^{-\frac{n}{2}\left(1-\frac{1}{p}\right)-\frac{1}{2}\left(1-\frac{1}{p}\right)} \norm{\mathbf{m}_0}_1 + e^{-c t} \norm{\mathbf{m}_0}_{H^{k_p(0)}}
\]
holds uniformly for all \(t \geq 1\).
\end{proposition}

\begin{proposition}\label{P3.3}
Let \(\mathbf{u}_0 = (a_0,\mathbf{m}_0)\). For any non-negative integer \(k\), there exists a constant \(\tilde{\varepsilon} > 0\) such that the inequalities below hold uniformly for all \(0 < t \leq 1\):
\begin{enumerate}
\item[(i)]
\[
\norm{\partial_{x'}^{k} \mathscr{U}(t) \mathbf{u}_0}_{H^{1} \times L^{2}} \lesssim \norm{\mathbf{u}_0}_{H^{k+1} \times H^{k}},
\]
\item[(ii)]
\[
\norm{\partial_{x'}^{k} \partial_{x} \mathscr{U}(t) \mathbf{u}_0}_{2} \lesssim t^{-\frac{1}{2}} \norm{\mathbf{u}_0}_{H^{k+1} \times H^{k}},
\]
\item[(iii)] in particular, when \(a_0 = 0\), i.e., \(\mathbf{u}_0 = (0,\mathbf{m}_0)\), we have
\[
\norm{\mathscr{U}(t)\mathbf{u}_0}_{\infty} \lesssim t^{-(1-\tilde{\varepsilon})} \norm{\mathbf{m}_0}_{H^{\left[\frac{n}{2}\right]}},
\]
where \(\tilde{\varepsilon}\) is a sufficiently small positive constant.
\end{enumerate}
\end{proposition}

We next present the second family of estimates, which describe the decay behavior of solutions to the heat equation on the half-space subject to homogeneous Neumann or Dirichlet boundary conditions.

\begin{proposition}\label{P3.4}
  Let \(n \geq 2\). Given integers \(k, j \in \mathbb{Z}\) and exponents \(1 \leq q \leq p \leq \infty\), the following estimates hold for all \(t > 0\).
  \begin{enumerate}
    \item[(i)]
    \[
    \norm{\partial_x^k \partial_t^j\mathscr{B}(t) \mathbf{B}_0}_p \lesssim t^{-\frac{n}{2}\left(\frac{1}{q}-\frac{1}{p}\right)-\frac{1}{2}(k+2j)} \norm{\mathbf{B}_0}_q.
    \]
    \item[(ii)] In particular, when \(\mathbf{B}_{0}\big|_{x_n=0} = \mathbf{0}\), one obtains
    \[
    \norm{\partial_x^k \partial_t^j\mathscr{B}(t) \partial_x\mathbf{B}_{0}}_p \lesssim t^{-\frac{n}{2}\left(\frac{1}{q}-\frac{1}{p}\right)-\frac{1}{2}(k+2j+1)} \norm{\mathbf{B}_{0}}_q.
    \]
    \item[(iii)]
    \[
    \norm{\partial_{x'}\partial_x^{k} \partial_t^j\mathscr{B}(t) \mathbf{B}_{0}}_p \lesssim t^{-\frac{n}{2}\left(\frac{1}{q}-\frac{1}{p}\right)-\frac{1}{2}(k+2j)} \norm{\partial_{x'}\mathbf{B}_{0}}_q.
    \]
  \end{enumerate}
\end{proposition}

\begin{proof}
We first treat the case of the perfectly conducting boundary condition
\(\left(\partial_{x_n}\mathbf{B}',B_n\right)\big|_{x_n=0} = \mathbf{0}\).

We split the magnetic field \(\mathbf{B}\) into its tangential and normal components and handle each component independently. The tangential component \(\mathbf{B}'\) obeys the heat equation with homogeneous Neumann boundary data, while the normal component \(B_n\) satisfies the heat equation with homogeneous Dirichlet boundary condition.
By carrying out even and odd extensions of \((\mathbf{B}_0)'\) and \((\mathbf{B}_0)_n\) respectively, we arrive at the following representation formula for \(\mathscr{B}(t)\mathbf{B}_0\):
\[
\left(\mathscr{B}(t) \mathbf{B}_0\right)' = \int_{\mathbb{R}^n_+} \left(G\left(t, \mathbf{x}'-\mathbf{y}', x_{n}-y_{n}\right) + G\left(t, \mathbf{x}'-\mathbf{y}', x_{n}+y_{n}\right)\right) \left(\mathbf{B}_0(\mathbf{y})\right)' \, \mathrm{d}\mathbf{y},
\]
\[
\left(\mathscr{B}(t) \mathbf{B}_0\right)_{n} = \int_{\mathbb{R}^n_+} \left(G\left(t, \mathbf{x}'-\mathbf{y}', x_{n}-y_{n}\right) - G\left(t, \mathbf{x}'-\mathbf{y}', x_{n}+y_{n}\right)\right) \left(\mathbf{B}_0(\mathbf{y})\right)_{n} \, \mathrm{d}\mathbf{y},
\]
where
\[
G(t, \mathbf{x}) \triangleq \frac{1}{(4\pi t)^{n/2}} \exp\left( -\frac{|\mathbf{x}|^2}{4t}\right)
\]
stands for the standard heat kernel on \(\mathbb{R}^n\). Estimate (i) then follows from Young’s convolution inequality combined with the \(L^r\)-integrability properties of the differentiated heat kernel \(\partial_x^k \partial_t^j G(t, \mathbf{x})\).

We next turn to estimate (ii). We focus on the term \(\norm{ \mathscr{B}(t)\partial_{x_n}\mathbf{B}_{0}}_p\); the remaining cases follow by analogous arguments. Integrating by parts and invoking the boundary condition \(\left(\mathbf{B}_{0}\right)'\big|_{x_n=0} = \mathbf{0}\), we obtain
\begin{equation}\label{3.3}
  \begin{aligned}
  &\left( \mathscr{B}(t)\partial_{x_n}\mathbf{B}_{0}\right)'\\
  ={}& \int_{\mathbb{R}^n_+} \left(G\left(t, \mathbf{x}'-\mathbf{y}', x_{n}-y_{n}\right) + G\left(t, \mathbf{x}'-\mathbf{y}', x_{n}+y_{n}\right)\right)  \partial_{y_n}\left(\mathbf{B}_{0}(\mathbf{y})\right)' \, \mathrm{d}\mathbf{y} \\
  ={}&-\int_{\mathbb{R}^n_+} \partial_{y_{n}} \left(G\left(t, \mathbf{x}'-\mathbf{y}', x_{n}-y_{n}\right) + G\left(t, \mathbf{x}'-\mathbf{y}', x_{n}+y_{n}\right)\right) \left(\mathbf{B}_{0}(\mathbf{y})\right)' \, \mathrm{d}\mathbf{y} \\
  ={}& \int_{\mathbb{R}^n_+}  \partial_{x_{n}}\left(G\left(t, \mathbf{x}'-\mathbf{y}', x_{n}-y_{n}\right) - G\left(t, \mathbf{x}'-\mathbf{y}', x_{n}+y_{n}\right)\right) \left(\mathbf{B}_{0}(\mathbf{y})\right)' \, \mathrm{d}\mathbf{y}.
  \end{aligned}
\end{equation}
Proceeding similarly via integration by parts, we deduce that
\begin{equation}\label{3.4}
  \begin{aligned}
  &\left(\mathscr{B}(t) \partial_{x_n}\mathbf{B}_{0}\right)_{n}\\
  ={}& \int_{\mathbb{R}^n_+} \left(G\left(t, \mathbf{x}'-\mathbf{y}', x_{n}-y_{n}\right) - G\left(t, \mathbf{x}'-\mathbf{y}', x_{n}+y_{n}\right)\right)\partial_{y_n}\left(\mathbf{B}_{0}(\mathbf{y})\right)_{n} \, \mathrm{d}\mathbf{y}\\
  ={}& \int_{\mathbb{R}^n_+} \partial_{y_{n}} \left(-G\left(t, \mathbf{x}'-\mathbf{y}', x_{n}-y_{n}\right) + G\left(t, \mathbf{x}'-\mathbf{y}', x_{n}+y_{n}\right)\right) \left(\mathbf{B}_{0}(\mathbf{y})\right)_{n} \, \mathrm{d}\mathbf{y} \\
  ={}& \int_{\mathbb{R}^n_+}  \partial_{x_{n}}\left(G\left(t, \mathbf{x}'-\mathbf{y}', x_{n}-y_{n}\right) + G\left(t, \mathbf{x}'-\mathbf{y}', x_{n}+y_{n}\right)\right)\left(\mathbf{B}_{0}(\mathbf{y})\right)_{n} \, \mathrm{d}\mathbf{y}.
  \end{aligned}
\end{equation}
Applying the same reasoning as in (i) to identities \eqref{3.3} and \eqref{3.4} yields estimate (ii).

The proof of (iii) follows along the same lines as that of (ii), and we therefore omit the details.

The case of the insulating boundary condition \(\left(\mathbf{B}', \partial_{x_n} B_n\right)\big|_{x_n=0} = \mathbf{0}\) can be handled by an identical argument, and we omit its proof for the sake of brevity.
\end{proof}

\subsection{\(H^1\) decay rates of global solutions}\label{{sub:3.2}}
We begin by establishing the \(L^2\) estimates stated in items (i), (ii) and (iv) of Theorem \ref{T1.3}.

For the convenience of the ensuing analysis, we introduce the following exponent parameters:
\[
\sigma_{0}\triangleq\frac{n}{4}+q(1),\quad \sigma_{1}\triangleq\frac{n}{4}-\varepsilon_1,\quad \sigma_{2}\triangleq\frac{n}{4}+q(1)-2\varepsilon_1,
\]
\[
\sigma_{3}\triangleq\frac{n}{4}+\frac{1}{2},\quad \sigma_{4}\triangleq\frac{n}{4}+\frac{1}{2}-2\varepsilon_1,\quad \sigma_{5}\triangleq\frac{n}{4}+1,
\]
where \(q(1)\) is the quantity defined in \eqref{3.1}, namely \(q(1)=\frac{3}{8}\) for \(n=2\) and \(q(1)=\frac{1}{2}\) for \(n \geq 3\), and \(\varepsilon_1>0\) is a sufficiently small constant.

We further define the following weighted supremum quantities:
\begin{align*}
  M_{\sigma}^{s}(t)\triangleq{}&\sup_{0 \leq \tau \leq t} (1+\tau)^{\sigma} \norm{\partial_{x}^{s} \mathbf{u}(\tau)}_{2},\\
B_{\sigma}^{s}(t)\triangleq{}&\sup_{0 \leq \tau \leq t} (1+\tau)^{\sigma} \norm{\partial_{x}^{s} \mathbf{B}(\tau)}_{2},\\
\mathscr{M}_{0}(t)\triangleq{}& M_{\frac{n}{4}}^{0}(t)+M_{\sigma_{0}}^{1}(t)+B_{\frac{n}{4}}^{0}(t)+B_{\sigma_{3}}^{1}(t).
\end{align*}

We shall prove that
\begin{equation}\label{3.5}
  \mathscr{M}_{0}(t) \lesssim \norm{(\mathbf{u}_0,\mathbf{B}_0)}_{1} +\norm{\mathbf{u}_0}_{H^{1} \times L^{2}}+ E_{0}^{s}(t)^{2} + D_{0}^{s}(t)^{2} + \mathscr{M}_{0}(t)^{2}
\end{equation}
for every \(t \geq 0\), under the condition that \(E_{0}^{s}(\infty) \leq 1\).
With this bound in hand, the desired \(H^1\) estimates in item (i) of Theorem \ref{T1.3} follow from Theorem \ref{T1.1} together with \eqref{3.5} by a standard argument, provided that \(\norm{(\mathbf{u}_0,\mathbf{B}_0)}_{H^s} + \norm{(\mathbf{u}_0,\mathbf{B}_0)}_{1} \leq \eta_{1}\) for some small constant \(\eta_{1}>0\).

In view of Theorem \ref{T1.1}, the short-time estimate \eqref{3.5} for \(t \leq 1\) is immediate. We may therefore assume \(t \geq 1\) throughout the rest of the proof.

We now decompose \(\mathbf{u}(t)\) and \(\mathbf{B}(t)\) via Duhamel’s principle, splitting the convolution integrals into short-time and long-time contributions:
\begin{align*}
\mathbf{u}(t) = {}& \mathscr{U}(t) \mathbf{u}_0 + \int_{t-1}^{t} \mathscr{U}(t-\tau) \operatorname{div} \begin{pmatrix} 0 \\ \mathscr{N}_1+\mathscr{N}_2 \end{pmatrix}(\tau) \mathrm{d}\tau \\
& + \int_{0}^{t-1} \mathscr{U}(t-\tau) \operatorname{div} \begin{pmatrix} 0 \\ \mathscr{N}_1+\mathscr{N}_2 \end{pmatrix}(\tau) \mathrm{d}\tau \\
\triangleq {}& I_{0}(t) + I_{1}(t) + I_{2}(t),\\
\mathbf{B}(t) = {}& \mathscr{B}(t) \mathbf{B}_0 + \int_{t-1}^{t} \mathscr{B}(t-\tau) \operatorname{div} \mathscr{N}_3(\tau) \mathrm{d}\tau \\
& + \int_{0}^{t-1} \mathscr{B}(t-\tau) \operatorname{div} \mathscr{N}_3(\tau) \mathrm{d}\tau \\
\triangleq{} & II_{0}(t) + II_{1}(t) + II_{2}(t).
\end{align*}

Applying item (i) of Proposition \ref{P3.1} and item (i) of Proposition \ref{P3.4} with \(\ell = 0, 1\), we arrive at
\begin{align*}
  \norm{\partial_{x}^{\ell} I_{0}(t)}_{2} \lesssim{}& t^{-\frac{n}{4}-q(\ell)} \norm{\mathbf{u}_0}_{1} + e^{-c t}\norm{\mathbf{u}_0}_{H^{1} \times L^{2}},\\
\norm{\partial_{x}^{\ell} II_{0}(t)}_{2} \lesssim{}& t^{-\frac{n}{4}-\frac{\ell}{2}}  \norm{\mathbf{B}_0}_{1},
\end{align*}
with \(q(\ell)\) given by \eqref{3.1}.

To bound the terms \(I_1(t)\) and \(I_2(t)\), we first collect several estimates for the nonlinear matrices \(\mathscr{N}_1\) and \(\mathscr{N}_2\), which are summarized in the proposition below.

\begin{proposition}\label{P3.5}
Let the regularity index satisfy \(s \geq s_0 + 1\), and suppose that \(E_0^s(t) \leq 1\) and \(a(t) \geq -\frac{1}{2}\) for every \(t \geq 0\). Then the inequalities below hold for all \(t \geq 0\):
\begin{enumerate}
\item[(i)]
\[
\norm{\partial_{x} \mathscr{N}_1(t)}_{2} \lesssim(1+t)^{-\sigma_{0}} \left( E_{0}^{s}(t)^{2}+  M_{\frac{n}{4}}^{0}(t)^{2} + M_{\sigma_{0}}^{1}(t)^{2} \right),
\]
\item[(ii)]
\begin{align*}
\norm{\partial_{x} \mathscr{N}_1(t)}_{1} \lesssim{}& (1+t)^{-\frac{n}{4}-\frac{\sigma_{0}}{2}} M_{\frac{n}{4}}^{0}(t) M_{\sigma_{0}}^{1}(t)^{\frac{1}{2}} \norm{\partial_{x}^{3} \mathbf{u}(t)}_{2}^{\frac{1}{2}} \\
& + (1+t)^{-\frac{n}{2}-q(1)} \left( M_{\frac{n}{4}}^{0}(t)^{2} + M_{\sigma_{0}}^{1}(t)^{2} \right),
\end{align*}
\item[(iii)]
\[
\norm{\mathscr{N}_1(t)}_{2} \lesssim(1+t)^{-\sigma_{0}} \left(E_{0}^{s}(t)^{2} + M_{\sigma_{0}}^{1}(t)^{2}\right),
\]
\item[(iv)]
\[
\norm{\mathscr{N}_1(t)}_{1} \lesssim(1+t)^{-\frac{n}{2}} \left( M_{\frac{n}{4}}^{0}(t)^{2} + M_{\sigma_{0}}^{1}(t)^{2} \right),
\]
\item[(v)]
\[
\norm{\partial_{x} \mathscr{N}_2(t)}_{2} \lesssim(1+t)^{-\sigma_{3}} \left(  E_{0}^{s}(t)^{2} +B_{\sigma_{3}}^{1}(t)^{2}\right),
\]
\item[(vi)]
\[
\norm{\partial_{x} \mathscr{N}_2(t)}_{1} \lesssim (1+t)^{-\frac{n}{4}-\sigma_{3}} \left(B_{\frac{n}{4}}^{0}(t)^{2}+B_{\sigma_{3}}^{1}(t)^{2}\right),
\]
\item[(vii)]
\[
\norm{\mathscr{N}_2(t)}_{2} \lesssim(1+t)^{-\sigma_{3}} \left( E_{0}^{s}(t)^{2}+B_{\sigma_{3}}^{1}(t)^{2} \right),
\]
\item[(viii)]
\[
\norm{\mathscr{N}_2(t)}_{1} \lesssim(1+t)^{-\frac{n}{2}}  B_{\frac{n}{4}}^{0}(t)^{2}.
\]
\end{enumerate}
\end{proposition}

\begin{proof}
Following the argument in \cite{Kagei-Kobayashi2005}, all these estimates follow from an application of Lemma \ref{L4.1} presented in the Appendix \ref{sec:4}; we therefore omit the full details of the derivation.
\end{proof}

We proceed to estimate the Duhamel terms \(I_1(t)\) and \(I_2(t)\). Combining items (i) and (ii) of Proposition \ref{P3.3} with items (i) and (v) of Proposition \ref{P3.5}, we deduce that for \(\ell = 0, 1\),
\begin{align*}
\norm{\partial_{x}^{\ell} I_{1}(t)}_{2}  \lesssim{}&\int_{t-1}^{t} (t-\tau)^{-\frac{\ell}{2}} \left(\norm{\operatorname{div} \mathscr{N}_1(\tau)}_{2} +\norm{\operatorname{div} \mathscr{N}_2(\tau)}_{2}\right)\mathrm{d}\tau \\
\lesssim{}&  \left( E_{0}^{s}(t)^{2}+ M_{\frac{n}{4}}^{0}(t)^{2} + M_{\sigma_{0}}^{1}(t)^{2} \right) \int_{t-1}^{t} (t-\tau)^{-\frac{\ell}{2}} (1+\tau)^{-\sigma_{0}} \mathrm{d}\tau \\
&+ \left( E_{0}^{s}(t)^{2}+ B_{\sigma_{3}}^{1}(t)^{2} \right)\int_{t-1}^{t} (t-\tau)^{-\frac{\ell}{2}} (1+\tau)^{-\sigma_{3}} \mathrm{d}\tau \\
 \lesssim{}&(1+t)^{-\sigma_{0}} \left( E_{0}^{s}(t)^{2} + \mathscr{M}_{0}(t)^{2} \right).
\end{align*}

Turning to the long-time contribution \(I_2(t)\), we invoke Proposition \ref{P3.2}, items (iii), (iv), (vii) and (viii) of Proposition \ref{P3.5}, and Lemma \ref{L4.4} to obtain
\begin{align*}
\norm{I_{2}(t)}_{2} \lesssim{}&  \int_{0}^{t-1} (t-\tau)^{-\frac{n}{4}-\frac{1}{4}} \left(\norm{\mathscr{N}_1(\tau)}_{1}+\norm{\mathscr{N}_2(\tau)}_{1}\right) \mathrm{d}\tau\\
& + \int_{0}^{t-1} e^{-c(t-\tau)} \left(\norm{\mathscr{N}_1(\tau)}_{2}+\norm{\mathscr{N}_2(\tau)}_{2}\right) \mathrm{d}\tau  \\
\lesssim{}& \left( M_{\frac{n}{4}}^{0}(t)^{2} + M_{\sigma_{0}}^{1}(t)^{2} + B_{\frac{n}{4}}^{0}(t)^{2} \right) \int_{0}^{t-1} (t-\tau)^{-\frac{n}{4}-\frac{1}{4}} (1+\tau)^{-\frac{n}{2}} \mathrm{d}\tau \\
& +   \left(E_{0}^{s}(t)^{2} + M_{\sigma_{0}}^{1}(t)^{2}\right) \int_{0}^{t-1} e^{-c(t-\tau)} (1+\tau)^{-\sigma_{0}} \mathrm{d}\tau \\
& +   \left(E_{0}^{s}(t)^{2}+ B_{\sigma_{3}}^{1}(t)^{2}  \right) \int_{0}^{t-1} e^{-c(t-\tau)} (1+\tau)^{-\sigma_{3}} \mathrm{d}\tau \\
\lesssim{}& (1+t)^{-\frac{n}{4}-\frac{1}{4}} L(t)\left( E_{0}^{s}(t)^{2} + \mathscr{M}_{0}(t)^{2} \right),
\end{align*}
where \(L(t) = \log(t+1)\) for \(n=2\) and \(L(t) = 1\) for \(n \geq 3\).

For the first-order spatial derivative \(\partial_x I_2(t)\), an application of item (i) of Proposition \ref{P3.1}, items (i), (ii), (v) and (vi) of Proposition \ref{P3.5}, and Lemma \ref{L4.4} yields
\begin{align*}
\norm{\partial_{x} I_{2}(t)}_{2} \lesssim{}&  \int_{0}^{t-1} (t-\tau)^{-\frac{n}{4}-q(1)} \left(\norm{\operatorname{div} \mathscr{N}_1(\tau)}_{1}+\norm{\operatorname{div} \mathscr{N}_2(\tau)}_{1}\right) \mathrm{d}\tau \\
& +  \int_{0}^{t-1} e^{-c(t-\tau)} \left(\norm{\operatorname{div} \mathscr{N}_1(\tau)}_{2}+\norm{\operatorname{div} \mathscr{N}_2(\tau)}_{2}\right) \mathrm{d}\tau \\
\lesssim{}&  M_{\frac{n}{4}}^{0}(t) M_{\sigma_{0}}^{1}(t)^{\frac{1}{2}} \\
& \cdot \left( \int_{0}^{t-1} (t-\tau)^{-\frac{4}{3}\left(\frac{n}{4}+q(1)\right)} (1+\tau)^{-\frac{4}{3}\left(\frac{n}{4}+\frac{\sigma_{0}}{2}\right)} \mathrm{d}\tau \right)^{\frac{3}{4}} \left( \int_{0}^{t-1} \norm{\partial_{x}^{3} \mathbf{u}}_{2}^{2} \mathrm{d}\tau \right)^{\frac{1}{4}} \\
& +  \left( M_{\frac{n}{4}}^{0}(t)^{2} + M_{\sigma_{0}}^{1}(t)^{2} \right) \int_{0}^{t-1} (t-\tau)^{-\frac{n}{4}-q(1)} (1+\tau)^{-\frac{n}{2}-q(1)}\mathrm{d}\tau \\
& +  \left(B_{\frac{n}{4}}^{0}(t)^{2}+B_{\sigma_{3}}^{1}(t)^{2}\right) \int_{0}^{t-1} (t-\tau)^{-\frac{n}{4}-q(1)} (1+\tau)^{-\frac{n}{4}-\sigma_{3}} \mathrm{d}\tau \\
& +  \left(E_{0}^{s}(t)^{2} + M_{\frac{n}{4}}^{0}(t)^{2} + M_{\sigma_{0}}^{1}(t)^{2}\right) \int_{0}^{t-1} e^{-c(t-\tau)} (1+\tau)^{-\sigma_{0}} \mathrm{d}\tau\\
& +  \left(E_{0}^{s}(t)^{2}+ B_{\sigma_{3}}^{1}(t)^{2} \right) \int_{0}^{t-1} e^{-c(t-\tau)} (1+\tau)^{-\sigma_{3}} \mathrm{d}\tau\\
\lesssim{}&(1+t)^{-\frac{n}{4}-q(1)} \left( E_{0}^{s}(t)^{2} + D_{0}^{s}(t)^{2} + \mathscr{M}_{0}(t)^{2} \right).
\end{align*}

To bound the terms \(II_1(t)\) and \(II_2(t)\), we next collect the corresponding estimates for the nonlinear matrix \(\mathscr{N}_3\), which are summarized in the proposition below.

\begin{proposition}\label{P3.6}
Let the regularity index satisfy \(s \geq s_0 + 1\), and suppose that \(E_0^s(t) \leq 1\) and \(a(t) \geq -\frac{1}{2}\) for every \(t \geq 0\). Then the inequalities below hold for all \(t \geq 0\):
\begin{enumerate}
\item[(i)]
\[
\norm{\partial_{x} \mathscr{N}_3(t)}_{2} \lesssim(1+t)^{-\sigma_{0}} \left( E_{0}^{s}(t)^{2}+M_{\sigma_{0}}^{1}(t)^{2} + B_{\sigma_{3}}^{1}(t)^{2} \right).
\]
\item[(ii)] In particular, for \(n=2\),
\[
\norm{\partial_{x} \mathscr{N}_3(t)}_{2} \lesssim (1+t)^{-\frac{9}{8}} \left(E_{0}^{s}(t)^{2}+ \mathscr{M}_{0}(t)^{2}\right).
\]
\item[(iii)]
\[
\norm{\partial_{x} \mathscr{N}_3(t)}_{1} \lesssim (1+t)^{-\frac{n}{4}-{\sigma_{0}}}  \mathscr{M}_{0}(t)^{2}.
\]
\item[(iv)]
\[
\norm{\mathscr{N}_3(t)}_{1} \lesssim(1+t)^{-\frac{n}{2}} \left( M_{\frac{n}{4}}^{0}(t)^{2} + B_{\frac{n}{4}}^{0}(t)^{2} \right).
\]
\end{enumerate}
\end{proposition}

\begin{proof}
Following the argument in \cite{Kagei-Kobayashi2005}, these estimates follow from an application of Lemma \ref{L4.1} presented in the Appendix \ref{sec:4}; we therefore omit the full details of the proof.
\end{proof}

We proceed to estimate the magnetic Duhamel terms \(II_1(t)\) and \(II_2(t)\). Combining item (i) of Proposition \ref{P3.4} with item (i) of Proposition \ref{P3.6}, we deduce that for \(\ell = 0, 1\),
\begin{align*}
\norm{\partial_{x}^{\ell} II_{1}(t)}_{2}  \lesssim{}&\int_{t-1}^{t} (t-\tau)^{-\frac{\ell}{2}} \norm{\operatorname{div} \mathscr{N}_3(\tau)}_{2} \mathrm{d}\tau \\
\lesssim{}&  \left(E_{0}^{s}(t)^{2}+ M_{\sigma_{0}}^{1}(t)^{2} + B_{\sigma_{3}}^{1}(t)^{2}\right) \int_{t-1}^{t} (t-\tau)^{-\frac{\ell}{2}} (1+\tau)^{-\sigma_{0}} \mathrm{d}\tau \\
 \lesssim{}&(1+t)^{-\sigma_{0}} \left( E_{0}^{s}(t)^{2} + \mathscr{M}_{0}(t)^{2} \right).
\end{align*}

In particular, for the two-dimensional case \(n = 2\), we invoke item (ii) of Proposition \ref{P3.6} together with item (i) of Proposition \ref{P3.4} to arrive at
\begin{align*}
\norm{\partial_{x}^{\ell} II_{1}(t)}_{2}  \lesssim{}&\int_{t-1}^{t} (t-\tau)^{-\frac{\ell}{2}} \norm{\operatorname{div} \mathscr{N}_3(\tau)}_{2} \mathrm{d}\tau \\
\lesssim{}&\left(E_{0}^{s}(t)^{2}+ \mathscr{M}_{0}(t)^{2} \right)\int_{t-1}^{t} (t-\tau)^{-\frac{1}{2}} (1+\tau)^{-\frac{9}{8}} \mathrm{d}\tau \\
 \lesssim{}&(1+t)^{-\frac{9}{8}} \left( E_{0}^{s}(t)^{2} + \mathscr{M}_{0}(t)^{2} \right).
\end{align*}

Turning to the long-time contribution \(II_2(t)\), we exploit the boundary condition \(\mathscr{N}_3\big|_{x_n=0} = \mathbf{0}\) and combine item (ii) of Proposition \ref{P3.4}, item (iv) of Proposition \ref{P3.6} and Lemma \ref{L4.4} to deduce that
\begin{align*}
\norm{II_{2}(t)}_{2} \lesssim{}&  \int_{0}^{t-1} (t-\tau)^{-\frac{n}{4}-\frac{1}{2}} \norm{\mathscr{N}_3(\tau)}_{1} \mathrm{d}\tau\\
\lesssim{}& \left( M_{\frac{n}{4}}^{0}(t)^{2} + B_{\frac{n}{4}}^{0}(t)^{2} \right) \int_{0}^{t-1} (t-\tau)^{-\frac{n}{4}-\frac{1}{2}} (1+\tau)^{-\frac{n}{2}}  \mathrm{d}\tau \\
\lesssim{}& (1+t)^{-\frac{n}{4}-\frac{1}{2}} L(t)\mathscr{M}_{0}(t)^{2}.
\end{align*}

For the first-order spatial derivative \(\partial_x II_2(t)\), an application of item (i) of Proposition \ref{P3.4}, item (iii) of Proposition \ref{P3.6} and Lemma \ref{L4.4} yields
\begin{align*}
\norm{\partial_{x} II_{2}(t)}_{2} \lesssim{}&  \int_{0}^{t-1} (t-\tau)^{-\frac{n}{4}-\frac{1}{2}} \norm{\operatorname{div} \mathscr{N}_3(\tau)}_{1}\mathrm{d}\tau \\
\lesssim{}&  \mathscr{M}_{0}(t)^{2}   \int_{0}^{t-1} (t-\tau)^{-\frac{n}{4}-\frac{1}{2}} (1+\tau)^{-\frac{n}{4}-{\sigma_{0}}} \mathrm{d}\tau \\
\lesssim{}&(1+t)^{-\frac{n}{4}-\frac{1}{2}} \mathscr{M}_{0}(t)^{2},
\end{align*}
which completes the proof of estimate \eqref{3.5}.

As noted earlier, the desired \(H^1\) estimates in item (i) of Theorem \ref{T1.3} and the bound \(\mathscr{M}_0(t) \lesssim \eta_1\) follow directly from Theorem \ref{T1.1} and inequality \eqref{3.5}, provided that \(\norm{(\mathbf{u}_0,\mathbf{B}_0)}_{H^s} + \norm{(\mathbf{u}_0,\mathbf{B}_0)}_1 \leq \eta_1\) for some sufficiently small constant \(\eta_1 > 0\).

With the foregoing decay estimates in hand, we further obtain
\begin{align*}
  \norm{\mathbf{u}(t)-\mathscr{U}(t) \mathbf{u}_{0}}_{2}\lesssim{}&\norm{I_{1}(t)}_{2}+\norm{I_{2}(t)}_{2} \lesssim (1+t)^{-\frac{n}{4}-\frac{1}{4}} L(t),\\
\norm{\mathbf{B}(t)-\mathscr{B}(t) \mathbf{B}_{0}}_{2}\lesssim{}&\norm{II_{1}(t)}_{2}+\norm{II_{2}(t)}_{2} \lesssim (1+t)^{-\frac{n}{4}-\frac{1}{2}} L(t),
\end{align*}
which establishes the \(L^2\) decay results stated in items (ii) and (iv) of Theorem \ref{T1.3}.

\subsection{\(L^{\infty}\) asymptotics and higher derivative decay}\label{sub:3.3}
We devote the present section to establishing decay estimates for higher-order derivatives by means of the weighted energy method. Building on these results, we further deduce the \(L^\infty\) estimates stated in items (i), (iii) and (iv) of Theorem \ref{T1.3}.
\subsubsection{Decay estimates for higher order derivatives}\label{ss:3.3.1}
Combining estimate \eqref{1.9}, inequality \eqref{2.14} and item (i) of Theorem \ref{T1.3}, we arrive at the proposition below, which directly yields the decay estimates for higher-order derivatives.

\begin{proposition}\label{P3.7}
Let the regularity index satisfy \(s \geq s_0 + 1\), and suppose that \(E_0^s(t) \leq 1\) for every \(t \geq 0\). If \(\norm{(\mathbf{u}_0,\mathbf{B}_0)}_{H^s} + \norm{(\mathbf{u}_0,\mathbf{B}_0)}_1 \leq \eta_1\), then the inequality below holds for all \(t \geq 0\):
\begin{equation}\label{3.6}
E_{\sigma_{1}}^{s}(t)+D_{\sigma_{1}}^{s}(t) \lesssim \eta_1.
\end{equation}
\end{proposition}
\subsubsection{\(L^{\infty}\) decay rates of global solutions}\label{ss:3.3.2}
We proceed to establish the \(L^\infty\) decay estimate stated in item (i) of Theorem \ref{T1.3}. Once this bound is established, the \(L^p\) estimates in item (i) of Theorem \ref{T1.3} for \(2 < p < \infty\) follow by interpolation, which completes the proof of Theorem \ref{T1.3} (i).

We derive the \(L^\infty\) estimate by combining the decay estimates for the linearized problem with the weighted energy bound \eqref{3.6}. We introduce the following weighted supremum quantities:
\begin{align*}
  M_{\infty, \frac{n}{2}}(t)\triangleq{}&\sup_{0 \leq \tau \leq t} (1+\tau)^{\frac{n}{2}} \norm{\mathbf{u}(\tau)}_{\infty},\\
B_{\infty, \frac{n}{2}}(t)\triangleq{}&\sup_{0 \leq \tau \leq t} (1+\tau)^{\frac{n}{2}} \norm{\mathbf{B}(\tau)}_{\infty},\\
\mathscr{M}_{\infty}(t)\triangleq{}& M_{\infty, \frac{n}{2}}(t)+B_{\infty, \frac{n}{2}}(t).
\end{align*}

Applying item (i) of Proposition \ref{P3.1} and item (i) of Proposition \ref{P3.4}, we arrive at
\begin{align*}
  \norm{I_{0}(t)}_{\infty} \lesssim{}&  t^{-\frac{n}{2}} \norm{\mathbf{u}_{0}}_{1} + e^{-c t}\norm{\mathbf{u}_{0}}_{H^{\left[\frac{n}{2}\right]+1} \times H^{\left[\frac{n}{2}\right]}},\\
\norm{II_{0}(t)}_{\infty} \lesssim{}&  t^{-\frac{n}{2}} \norm{\mathbf{B}_{0}}_{1}.
\end{align*}

To bound the terms \(I_1(t)\) and \(I_2(t)\), we first collect the corresponding estimates for the nonlinearities \(\mathscr{N}_1\) and \(\mathscr{N}_2\), which are summarized in the proposition below.

\begin{proposition}\label{P3.8}
Let the regularity index satisfy \(s \geq s_0 + 1\), and suppose that \(E_0^s(t) \leq 1\) and \(a(t) \geq -\frac{1}{2}\) for every \(t \geq 0\). Then the inequalities below hold for all \(t \geq 0\):
\begin{enumerate}
\item[(i)]
\[
\norm{\partial_{x} \mathscr{N}_{1}(t)}_{H^{\left[\frac{n}{2}\right]}} \lesssim (1+t)^{-\frac{n}{2}} \left(  E_{\sigma_{1}}^{s}(t)^{2} + M_{\sigma_{0}}^{1}(t)^{2} +M_{\infty, \frac{n}{2}}(t)^{2}\right),
\]
\item[(ii)]
\[
\norm{\partial_{x} \mathscr{N}_{1}(t)}_{1} \lesssim (1+t)^{-\frac{n}{2}-\frac{1}{2}\left(q(1)-\varepsilon_{1}\right)} \left( E_{\sigma_{1}}^{s}(t)^{2} +\mathscr{M}_{0}(t)^{2}\right),
\]
\item[(iii)]
\[
\norm{\partial_{x} \mathscr{N}_{2}(t)}_{H^{\left[\frac{n}{2}\right]}} \lesssim (1+t)^{-\frac{n}{2}} \left( E_{\sigma_{1}}^{s}(t)^{2} + B_{\sigma_{3}}^{1}(t)^{2}+B_{\infty, \frac{n}{2}}(t)^{2} \right).
\]
\end{enumerate}
\end{proposition}

\begin{proof}
Following the argument in \cite{Kagei-Kobayashi2005}, the inequalities in Proposition \ref{P3.8} follow from a direct application of Lemmas \ref{L4.1}–\ref{L4.3} presented in the Appendix \ref{sec:4}.
\end{proof}

We proceed to estimate the Duhamel terms \(I_1(t)\) and \(I_2(t)\).
Combining item (iii) of Proposition \ref{P3.3} with items (i) and (iii) of Proposition \ref{P3.8}, we deduce that
\begin{align*}
\norm{I_{1}(t)}_{\infty} \lesssim{}&  \int_{t-1}^{t} (t-\tau)^{-(1-\tilde{\varepsilon})} \left(\norm{\operatorname{div} \mathscr{N}_{1}(\tau)}_{H^{\left[\frac{n}{2}\right]}}+\norm{\operatorname{div} \mathscr{N}_{2}(\tau)}_{H^{\left[\frac{n}{2}\right]}}\right) \mathrm{d}\tau \\
 \lesssim{}&  \left(E_{\sigma_{1}}^{s}(t)^{2} + M_{\sigma_{0}}^{1}(t)^{2}+ M_{\infty, \frac{n}{2}}(t)^{2} + B_{\sigma_{3}}^{1}(t)^{2} + B_{\infty, \frac{n}{2}}(t)^{2} \right)\\
 &\cdot  \int_{t-1}^{t} (t-\tau)^{-(1-\tilde{\varepsilon})} (1+\tau)^{-\frac{n}{2}} \mathrm{d}\tau \\
 \lesssim{}& (1+t)^{-\frac{n}{2}} \left( E_{\sigma_{1}}^{s}(t)^{2}+ \mathscr{M}_{0}(t)^{2}+\mathscr{M}_{\infty}(t)^{2} \right).
\end{align*}

Turning to the long-time contribution \(I_2(t)\), we invoke item (i) of Proposition \ref{P3.1}, item (vi) of Proposition \ref{P3.5}, Proposition \ref{P3.8} and Lemma \ref{L4.4} to obtain
\begin{align*}
\norm{I_{2}(t)}_{\infty} \lesssim{}& \int_{0}^{t-1} (t-\tau)^{-\frac{n}{2}} \left(\norm{\operatorname{div} \mathscr{N}_{1}(\tau)}_{1}+\norm{\operatorname{div} \mathscr{N}_{2}(\tau)}_{1}\right) \mathrm{d}\tau\\
& + \int_{0}^{t-1} e^{-c(t-\tau)} \left(\norm{\operatorname{div} \mathscr{N}_{1}(\tau)}_{H^{\left[\frac{n}{2}\right]}}+\norm{\operatorname{div} \mathscr{N}_{2}(\tau)}_{H^{\left[\frac{n}{2}\right]}}\right) \mathrm{d}\tau \\
\lesssim{}& \left( E_{\sigma_{1}}^{s}(t)^{2} + \mathscr{M}_{0}(t)^{2}\right) \int_{0}^{t-1} (t-\tau)^{-\frac{n}{2}} (1+\tau)^{-\frac{n}{2}-\frac{1}{2}\left(q(1)-\varepsilon_{1}\right)} \mathrm{d}\tau \\
&+ \left(B_{\frac{n}{4}}^{0}(t)^{2}+B_{\sigma_{3}}^{1}(t)^{2}\right)\int_{0}^{t-1} (t-\tau)^{-\frac{n}{2}} (1+\tau)^{-\frac{n}{4}-\sigma_{3}} \mathrm{d}\tau \\
&+ \left(E_{\sigma_{1}}^{s}(t)^{2} + M_{\sigma_{0}}^{1}(t)^{2} + M_{\infty, \frac{n}{2}}(t)^{2}+B_{\sigma_{3}}^{1}(t)^{2} + B_{\infty, \frac{n}{2}}(t)^{2}\right)\\
 &\cdot  \int_{0}^{t-1} e^{-c(t-\tau)} (1+\tau)^{-\frac{n}{2}} \mathrm{d}\tau \\
\lesssim{}&(1+t)^{-\frac{n}{2}} \left(E_{\sigma_{1}}^{s}(t)^{2} + \mathscr{M}_{0}(t)^{2}+ \mathscr{M}_{\infty}(t)^{2} \right).
\end{align*}

To bound the terms \(II_1(t)\) and \(II_2(t)\), we first collect the corresponding estimate for the nonlinearity \(\mathscr{N}_3\), which is stated in the proposition below.

\begin{proposition}\label{P3.9}
Let the regularity index satisfy \(s \geq s_0 + 1\), and suppose that \(E_0^s(t) \leq 1\) and \(a(t) \geq -\frac{1}{2}\) for every \(t \geq 0\). Then the inequality below holds for all \(t \geq 0\):
\[
\norm{\mathscr{N}_{3}(t)}_{\infty} \lesssim (1+t)^{-n} \mathscr{M}_{\infty}(t)^{2}.
\]
\end{proposition}

\begin{proof}
The desired estimate follows from a direct application of Hölder’s inequality.
\end{proof}

We proceed to estimate the magnetic Duhamel terms \(II_1(t)\) and \(II_2(t)\).
Exploiting the boundary condition \(\mathscr{N}_3\big|_{x_n=0} = \mathbf{0}\) and combining item (ii) of Proposition \ref{P3.4} with Proposition \ref{P3.9}, we deduce that
\begin{align*}
\norm{II_{1}(t)}_{\infty} \lesssim{}&  \int_{t-1}^{t} (t-\tau)^{-\frac{1}{2}} \norm{ \mathscr{N}_{3}(\tau)}_{\infty} \mathrm{d}\tau \\
\lesssim{}&  \mathscr{M}_{\infty}(t)^{2} \int_{t-1}^{t} (t-\tau)^{-\frac{1}{2}} (1+\tau)^{-n} \mathrm{d}\tau \\
\lesssim{}& (1+t)^{-n} \mathscr{M}_{\infty}(t)^{2}.
\end{align*}

Turning to the long-time contribution \(II_2(t)\), we again use the boundary condition \(\mathscr{N}_3\big|_{x_n=0} = \mathbf{0}\) together with items (i) and (ii) of Proposition \ref{P3.4} and items (iii) and (iv) of Proposition \ref{P3.6} to obtain
\begin{align*}
\norm{II_{2}(t)}_{\infty} \lesssim{} & \int_{0}^{\frac{t-1}{2}} (t-\tau)^{-\frac{n}{2}-\frac{1}{2}} \norm{\mathscr{N}_{3}(\tau)}_{1} \mathrm{d}\tau+\int^{t-1}_{\frac{t-1}{2}} (t-\tau)^{-\frac{n}{2}} \norm{\operatorname{div}\mathscr{N}_{3}(\tau)}_{1} \mathrm{d}\tau\\
\lesssim{} & \mathscr{M}_{0}(t)^{2}\left(\int_{0}^{\frac{t-1}{2}} (t-\tau)^{-\frac{n}{2}-\frac{1}{2}} (1+\tau)^{-\frac{n}{2}} \mathrm{d}\tau+\int^{t-1}_{\frac{t-1}{2}} (t-\tau)^{-\frac{n}{2}} (1+\tau)^{-\frac{n}{2}-q(1)} \mathrm{d}\tau\right) \\
\lesssim{} & (1+t)^{-\frac{n}{2}-q(1)}L(t) \mathscr{M}_{0}(t)^{2}.
\end{align*}

Gathering all the foregoing estimates, we conclude that
\[
\mathscr{M}_{\infty}(t) \lesssim  \norm{(\mathbf{u}_{0},\mathbf{B}_{0})}_{1}+\norm{\mathbf{u}_{0}}_{H^{\left[\frac{n}{2}\right]+1} \times H^{\left[\frac{n}{2}\right]}}+ E_{\sigma_{1}}^{s}(t)^{2} + \mathscr{M}_{0}(t)^{2} + \mathscr{M}_{\infty}(t)^{2}.
\]

As a consequence, the bound \(\mathscr{M}_\infty(t) \lesssim \eta_1\) holds provided that \(\norm{(\mathbf{u}_0,\mathbf{B}_0)}_{H^s} + \norm{(\mathbf{u}_0,\mathbf{B}_0)}_1 \leq \eta_1\) for some sufficiently small constant \(\eta_1 > 0\).

With these bounds in hand, we further deduce that
\[
\norm{\mathbf{B}(t)-\mathscr{B}(t) \mathbf{B}_{0}}_{\infty}\lesssim\norm{II_{1}(t)}_{\infty}+\norm{II_{2}(t)}_{\infty} \lesssim (1+t)^{-\frac{n}{2}-q(1)} L(t).
\]

This completes the proof of the \(L^\infty\) decay estimates stated in items (i) and (iv) of Theorem \ref{T1.3}.

\subsubsection{Improved decay estimates for higher order derivatives}\label{ss:3.3.3}
Before stating the main results of this subsection, we first introduce the following key energy and dissipation functionals:
\begin{align*}
   \widetilde{E}_{\sigma_2 }^{s}(t)\triangleq {}& \sup_{0\leq \tau \leq t}(1+\tau)^{\sigma_2 }\bigg(\norm{\partial_{x'}(a, \mathbf{v})(\tau)}_{2}+\norm{\partial_{x}^2{a}(\tau)}_{2}+\norm{\partial_{x'}\partial_{x}\mathbf{v}(\tau)}_{2}+\norm{\partial_x^3(a, \mathbf{v})(\tau)}_{H^{s-3}}\bigg) \\
     &+\sum_{j=1}^{\left[\frac{s+1}{2}\right]}\sup_{0\leq \tau \leq t}(1+\tau)^{\sigma_2 } \norm{\partial_{\tau}^{j} a(\tau)}_{H^{s+1-2 j}} + \sum_{j=1}^{\left[\frac{s}{2}\right]}\sup_{0\leq \tau \leq t} (1+\tau)^{\sigma_2 }\norm{\partial_{\tau}^{j} \mathbf{v}(\tau)}_{H^{s-2 j}},\\
\widetilde{D}_{\sigma_2 }^{s}(t)\triangleq{}& \Bigg(\int_{0}^{t}(1+\tau)^{2\sigma_2 }\bigg(\norm{\partial_{x'}\partial_{x}\mathbf{v}(\tau)}_{2}^{2}+\norm{\partial_{x}^2 a (\tau)}_{H^{s-2}}^{2}+\norm{\partial_{x}^3\mathbf{v}(\tau)}_{H^{s-2}}^{2} \\
&+\norm{\partial_x\partial_{\tau} (a, \mathbf{v})(\tau)}_{H^{s-2}}^{2}\bigg) \mathrm{d} \tau+\sum_{j=2}^{\left[\frac{s+1}{2}\right]} \int_{0}^{t}(1+\tau)^{2\sigma_2 }\norm{\partial_{\tau}^{j} (a, \mathbf{v})(\tau)}_{H^{s+1-2 j}}^{2} \mathrm{d} \tau \Bigg)^{\frac{1}{2}},\\
  \widetilde{M}_{\sigma_0}^{3}(t)\triangleq{}&  \sup_{0\leq \tau \leq t}(1+\tau)^{\sigma_0 }\left(\norm{\partial_{x}^2\mathbf{u}(\tau)}_{2}+\norm{\partial_{\tau}\mathbf{u}(\tau)}_{H^1}+\norm{\partial_{x'}\partial_{x}^2\mathbf{m}(\tau)}_{2}\right),\\
   \widetilde{F}_{\sigma_4 }^{s}(t)\triangleq{} & \sup_{0\leq \tau \leq t}(1+\tau)^{\sigma_4 }\left(\norm{\partial_{x'}\mathbf{B}(\tau)}_{2}+\norm{\partial_{x'}\partial_{x}\mathbf{B}(\tau)}_{2}+\norm{\partial_x^3 \mathbf{B}(\tau)}_{H^{s-3}}\right) \\
    & + \sum_{j=1}^{\left[\frac{s}{2}\right]}\sup_{0\leq \tau \leq t} (1+\tau)^{\sigma_4 }\norm{\partial_{\tau}^{j} \mathbf{B}(\tau)}_{H^{s-2 j}},\\
\widetilde{G}_{\sigma_4 }^{s}(t)\triangleq{}& \Bigg( \int_{0}^{t}(1+\tau)^{2\sigma_4 }\bigg(\norm{\partial_{x'}\partial_{x}\mathbf{B}(\tau)}_{2}^{2}+\norm{\partial_{x}^3\mathbf{B}(\tau)}_{H^{s-2}}^{2} +\norm{\partial_x\partial_{\tau}  \mathbf{B}(\tau)}_{H^{s-2}}^{2}\bigg) \mathrm{d} \tau\\
&+\sum_{j=2}^{\left[\frac{s+1}{2}\right]} \int_{0}^{t}(1+\tau)^{2\sigma_4 }\norm{\partial_{\tau}^{j} \mathbf{B}(\tau)}_{H^{s+1-2 j}}^{2} \mathrm{d} \tau \Bigg)^{\frac{1}{2}},\\
\widetilde{B}_{\sigma_5}^{2}(t)\triangleq{}&  \sup_{0\leq \tau \leq t}(1+\tau)^{\sigma_5}\left(\norm{\partial_{x_n}^2\mathbf{B}(\tau)}_{2}+\norm{\partial_{\tau}\mathbf{B}(\tau)}_{2}\right),\\
\mathscr{M}_{1}(t)\triangleq{}&  M_{\sigma_0}^{1}(t)+\widetilde{M}_{\sigma_0}^{3}(t)+\widetilde{E}_{\sigma_2}^{s}(t)+B_{\sigma_3}^{1}(t)+\widetilde{B}_{\sigma_5}^{2}(t)+\widetilde{F}_{\sigma_4}^{s}(t).
\end{align*}

We next proceed to refine the decay estimates for higher-order derivatives, which are formulated in the propositions below. To ease the presentation, we first define the following nonlinear remainder terms:
\begin{align*}
\widetilde{N}_{\sigma_2}^{s}(t)\triangleq{}&\Bigg((1+t)^{2\sigma_2 }\left( \norm{\partial_{x}^{2}\left( a  \operatorname{div} \mathbf{v}\right)}_{H^{s-3}}^{2}+\norm{\partial_{x}^{2} \left(\mathbf{v} \cdot \nabla  a\right)}_{H^{s-3}}^{2}\right)\\
& +\sum_{j=1}^{\left[\frac{s-1}{2}\right]}(1+t)^{2\sigma_2 }\left( \norm{\partial_{t}^{j}\left( a  \operatorname{div} \mathbf{v}\right)}_{H^{s-1-2j}}^{2}+\norm{\partial_{t}^{j} \left(\mathbf{v} \cdot \nabla  a\right)}_{H^{s-1-2j}}^{2}\right)\\
& +(1+t)^{2\sigma_2 } \norm{\partial_{x}{f}_2}_{H^{s-3}}^{2}+\sum_{j=1}^{\left[\frac{s-2}{2}\right]}(1+t)^{2\sigma_2 } \norm{\partial_{t}^{j}{f}_2}_{H^{s-2-2j}}^{2}\\
&+ \int_{0}^{t} (1+\tau)^{2\sigma_2 }\left(\norm{\partial_{x}\left( a  \operatorname{div} \mathbf{v}\right)}_{H^{s-1}}^{2}+\sum_{j=1}^{\left[\frac{s}{2}\right]}\norm{\partial_{\tau}^{j}\left( a  \operatorname{div} \mathbf{v}\right)}_{H^{s-2j}}^{2}\right) \mathrm{d} \tau\\
& +\int_{0}^{t} (1+\tau)^{2\sigma_2 }\left(\norm{\partial_{x}^{2} \left(\mathbf{v} \cdot \nabla  a\right)}_{H^{s-3}}^{2}+\sum_{j=1}^{\left[\frac{s-1}{2}\right]}\norm{\partial_{\tau}^{j} \left(\mathbf{v} \cdot \nabla  a\right)}_{H^{s-1-2j}}^{2}\right) \mathrm{d} \tau \\
& +\int_{0}^{t} (1+\tau)^{2\sigma_2 }\left(\norm{\partial_{x}{f}_2}_{H^{s-2}}^{2}+\sum_{j=1}^{\left[\frac{s-1}{2}\right]}\norm{\partial_{\tau}^{j} {f}_2}_{H^{s-1-2j}}^{2}\right) \mathrm{d} \tau \\
& +\sum_{1 \leq k+2 j \leq s} \int_{0}^{t}  (1+\tau)^{2\sigma_2 }  \abs{A_{k,j}}  \mathrm{d} \tau + \sum_{1 \leq k+2 j \leq s-1} \int_{0}^{t}(1+\tau)^{2\sigma_2 }\abs{B_{k,j}}\mathrm{d} \tau\\
& +\sum_{1 \leq k+\ell+2 j \leq s-1} \int_{0}^{t} (1+\tau)^{2\sigma_2 }\abs{C_{k,\ell,j}} \mathrm{d} \tau\Bigg)^{\frac{1}{2}}, \\
\widetilde{H}_{\sigma_4}^{s}(t)\triangleq{}&\Bigg((1+t)^{2\sigma_4 } \norm{\partial_{x}{f}_3}_{H^{s-3}}^{2}+\sum_{j=1}^{\left[\frac{s-2}{2}\right]}(1+t)^{2\sigma_4 } \norm{\partial_{t}^{j}{f}_3}_{H^{s-2-2j}}^{2}\\
& +\int_{0}^{t} (1+\tau)^{2\sigma_4 }\left(\norm{\partial_{x}{f}_3}_{H^{s-2}}^{2}+\sum_{j=1}^{\left[\frac{s-1}{2}\right]}\norm{\partial_{\tau}^{j} {f}_3}_{H^{s-1-2j}}^{2}\right) \mathrm{d} \tau \\
& +\sum_{1 \leq k+2 j \leq s} \int_{0}^{t}  (1+\tau)^{2\sigma_4 }  \abs{D_{k,j}}  \mathrm{d} \tau + \sum_{1 \leq k+2 j \leq s-1} \int_{0}^{t}(1+\tau)^{2\sigma_4 }\abs{E_{k,j}}\mathrm{d} \tau\Bigg)^{\frac{1}{2}},
\end{align*}

We then have the following result.

\begin{proposition}\label{P3.10}
Let the regularity index satisfy \(s \geq s_0 + 2\), and suppose that \(E_0^s(t) \leq 1\) for every \(t \geq 0\).
Then the inequalities below hold for all \(t \geq 0\):
\begin{equation}\label{3.7}
\begin{aligned}
\widetilde{E}_{\sigma_2 }^{s}(t)^{2}+\widetilde{D}_{\sigma_2 }^{s}(t)^{2} \lesssim{}& \norm{( a _0,\mathbf{v}_0)}_{H^s}^2+\widetilde{N}_{\sigma_2 }^{s}(t)^{2}+\sigma_2 \bigg(\delta \widetilde{D}_{\sigma_2 }^{s}(t)^{2}+D_{0}^{s}(t)^{2}\\
&+\int_{0}^{t}  (1+\tau)^{2\sigma_2 -1} \left(\norm{\partial_{x'}({a}, \mathbf{v})}_{2}^{2}+\norm{\partial_{\tau}({a}, \mathbf{v})}_{2}^{2}\right) \mathrm{d} \tau\bigg),
\end{aligned}
\end{equation}
\begin{equation}\label{3.8}
\begin{aligned}
\widetilde{F}_{\sigma_4 }^{s}(t)^{2}+\widetilde{G}_{\sigma_4 }^{s}(t)^{2} \lesssim{}& \norm{\mathbf{B}_0}_{H^s}^2+\widetilde{H}_{\sigma_4 }^{s}(t)^{2}+\sigma_4 \bigg(\delta \widetilde{G}_{\sigma_4 }^{s}(t)^{2}+D_{0}^{s}(t)^{2}\\
&+\int_{0}^{t}  (1+\tau)^{2\sigma_4 -1} \left(\norm{\partial_{x'}\mathbf{B}}_{2}^{2}+\norm{\partial_{\tau} \mathbf{B}}_{2}^{2}\right) \mathrm{d} \tau\bigg)
\end{aligned}
\end{equation}
for any \(\delta>0\).
\end{proposition}

\begin{proof}
Proceeding along the same lines as in the proof of Proposition \ref{P2.2}, we first establish a sequence of intermediate estimates.
\begin{align*}
& \sum_{1 \leq k+2 j \leq s}\left((1+t)^{2\sigma_2 }\norm{\partial_{x'}^{k} \partial_{t}^{j} (a,\mathbf{v})(t)}_{2}^{2}+\int_{0}^{t}  (1+\tau)^{2\sigma_2 }\norm{\partial_{x'}^{k} \partial_{\tau}^{j} \left(\frac{D a}{D\tau},\partial_x\mathbf{v}\right)}_{2}^{2}  \mathrm{d} \tau \right)\\
\lesssim {}& \norm{( a _0,\mathbf{v}_0)}_{H^s}^2+\widetilde{N}_{\sigma_2 }^{s}(t)^{2}+\sigma_2 \bigg(\delta \widetilde{D}_{\sigma_2 }^{s}(t)^{2}+D_{0}^{s}(t)^{2}+\int_{0}^{t}  (1+\tau)^{2\sigma_2 -1} \Big(\norm{\partial_{x'}({a}, \mathbf{v})}_{2}^{2}+\norm{\partial_{\tau}({a}, \mathbf{v})}_{2}^{2}\Big) \mathrm{d} \tau\bigg),\\
& \sum_{1 \leq k+2j \leq s-1}\left((1+t)^{2\sigma_2 }\norm{\partial_{x}\partial_{x'}^{k} \partial_{t}^{j} \mathbf{v}(t)}_{2}^{2}+\int_{0}^{t} (1+\tau)^{2\sigma_2 } \norm{\partial_{x'}^{k} \partial_{\tau}^{j+1} (a, \mathbf{v})}_{2}^{2}  \mathrm{d} \tau\right) \\
\lesssim {}& \norm{( a _0,\mathbf{v}_0)}_{H^s}^2+\widetilde{N}_{\sigma_2 }^{s}(t)^{2}+\sigma_2 \bigg(\delta \widetilde{D}_{\sigma_2 }^{s}(t)^{2}+D_{0}^{s}(t)^{2}+\int_{0}^{t}  (1+\tau)^{2\sigma_2 -1} \Big(\norm{\partial_{x'}({a}, \mathbf{v})}_{2}^{2}+\norm{\partial_{\tau}({a}, \mathbf{v})}_{2}^{2}\Big) \mathrm{d} \tau\bigg),\\
& \sum_{1 \leq k+\ell+2j \leq s-1}\left((1+t)^{2\sigma_2 }\norm{\partial_{x'}^{k} \partial_{x_n}^{\ell+1} \partial_{t}^{j}  a (t)}_{2}^{2}+\int_{0}^{t} (1+\tau)^{2\sigma_2 } \norm{\partial_{x'}^{k} \partial_{x_n}^{\ell+1} \partial_{\tau}^{j} \left(a, \frac{D a }{D\tau}\right)}_{2}^{2}  \mathrm{d} \tau\right) \\
\lesssim {}& \norm{( a _0,\mathbf{v}_0)}_{H^s}^2+\widetilde{N}_{\sigma_2 }^{s}(t)^{2}+\sigma_2 \bigg(\delta \widetilde{D}_{\sigma_2 }^{s}(t)^{2}+D_{0}^{s}(t)^{2}+\int_{0}^{t}  (1+\tau)^{2\sigma_2 -1} \Big(\norm{\partial_{x'}({a}, \mathbf{v})}_{2}^{2}+\norm{\partial_{\tau}({a}, \mathbf{v})}_{2}^{2}\Big) \mathrm{d} \tau\bigg),\\
 &\sum_{1 \leq k+\ell+2j \leq s-2}(1+t)^{2\sigma_2 }\norm{\partial_{x'}^{k} \partial_{x_n}^{\ell+2} \partial_{t}^{j} \mathbf{v}(t)}_{2}^{2}\\
\lesssim {}& \norm{( a _0,\mathbf{v}_0)}_{H^s}^2+\widetilde{N}_{\sigma_2 }^{s}(t)^{2}+\sigma_2 \bigg(\delta \widetilde{D}_{\sigma_2 }^{s}(t)^{2}+D_{0}^{s}(t)^{2}+\int_{0}^{t}  (1+\tau)^{2\sigma_2 -1} \Big(\norm{\partial_{x'}({a}, \mathbf{v})}_{2}^{2}+\norm{\partial_{\tau}({a}, \mathbf{v})}_{2}^{2}\Big) \mathrm{d} \tau\bigg),\\
&\sum_{1 \leq k+\ell+2j \leq s-1}\int_{0}^{t}(1+\tau)^{2\sigma_2 }\left(\norm{\partial_{x}^{\ell+1} \partial_{x'}^{k} \partial_{\tau}^{j}  a }_{2}^{2}+\norm{\partial_{x}^{\ell+2} \partial_{x'}^{k} \partial_{\tau}^{j} \mathbf{v}}_{2}^{2}\right)  \mathrm{d} \tau\\
\lesssim {}& \norm{( a _0,\mathbf{v}_0)}_{H^s}^2+\widetilde{N}_{\sigma_2 }^{s}(t)^{2}+\sigma_2 \bigg(\delta \widetilde{D}_{\sigma_2 }^{s}(t)^{2}+D_{0}^{s}(t)^{2}+\int_{0}^{t}  (1+\tau)^{2\sigma_2 -1} \Big(\norm{\partial_{x'}({a}, \mathbf{v})}_{2}^{2}+\norm{\partial_{\tau}({a}, \mathbf{v})}_{2}^{2}\Big) \mathrm{d} \tau\bigg),\\
& \sum_{2 \leq k+2j \leq s-1}\left((1+t)^{2\sigma_2 }\norm{\partial_{x}^{k} \partial_{t}^{j+1}  a (t)}_{2}^{2}+\int_{0}^{t} (1+\tau)^{2\sigma_2 } \norm{\partial_{x}^{k} \partial_{\tau}^{j+1}  a }_{2}^{2}  \mathrm{d} \tau\right) \\
\lesssim {}& \norm{( a _0,\mathbf{v}_0)}_{H^s}^2+\widetilde{N}_{\sigma_2 }^{s}(t)^{2}+\sigma_2 \bigg(\delta \widetilde{D}_{\sigma_2 }^{s}(t)^{2}+D_{0}^{s}(t)^{2}+\int_{0}^{t}  (1+\tau)^{2\sigma_2 -1} \Big(\norm{\partial_{x'}({a}, \mathbf{v})}_{2}^{2}+\norm{\partial_{\tau}({a}, \mathbf{v})}_{2}^{2}\Big) \mathrm{d} \tau\bigg).
\end{align*}
Summing all these intermediate estimates yields inequality \eqref{3.7}.

By an analogous argument for the magnetic field, we obtain the corresponding estimates:
\begin{align*}
& \sum_{1 \leq k+2 j \leq s}\left((1+t)^{2\sigma_4 }\norm{\partial_{x'}^{k} \partial_{t}^{j} \mathbf{B}(t)}_{2}^{2}+\int_{0}^{t} (1+\tau)^{2\sigma_4 } \norm{\partial_{x}\partial_{x'}^{k} \partial_{\tau}^{j}\mathbf{B}}_{2}^{2}  \mathrm{d} \tau\right) \\
\lesssim{}& \norm{\mathbf{B}_0}_{H^s}^2+\widetilde{H}_{\sigma_4 }^{s}(t)^{2}+\sigma_4 \bigg(\delta \widetilde{G}_{\sigma_4 }^{s}(t)^{2}+D_{0}^{s}(t)^{2}+\int_{0}^{t}  (1+\tau)^{2\sigma_4 -1} \left(\norm{\partial_{x'}\mathbf{B}}_{2}^{2}+\norm{\partial_{\tau} \mathbf{B}}_{2}^{2}\right) \mathrm{d} \tau\bigg),\\
& \sum_{1 \leq k+2 j \leq s-1}\left((1+t)^{2\sigma_4 } \norm{\partial_{x}\partial_{x'}^{k} \partial_{t}^{j} \mathbf{B}(t)}_{2}^{2}+\int_{0}^{t}(1+\tau)^{2\sigma_4 }  \norm{\partial_{x'}^{k} \partial_{\tau}^{j+1} \mathbf{B}}_{2}^{2}  \mathrm{d} \tau \right) \\
\lesssim{}& \norm{\mathbf{B}_0}_{H^s}^2+\widetilde{H}_{\sigma_4 }^{s}(t)^{2}+\sigma_4 \bigg(\delta \widetilde{G}_{\sigma_4 }^{s}(t)^{2}+D_{0}^{s}(t)^{2}+\int_{0}^{t}  (1+\tau)^{2\sigma_4 -1} \left(\norm{\partial_{x'}\mathbf{B}}_{2}^{2}+\norm{\partial_{\tau} \mathbf{B}}_{2}^{2}\right) \mathrm{d} \tau\bigg),\\
 &\sum_{1 \leq k+\ell+2j \leq s-2}(1+t)^{2\sigma_4 }\norm{\partial_{x'}^{k} \partial_{x_n}^{\ell+2} \partial_{t}^{j} \mathbf{B}(t)}_{2}^{2}\\
\lesssim{}& \norm{\mathbf{B}_0}_{H^s}^2+\widetilde{H}_{\sigma_4 }^{s}(t)^{2}+\sigma_4 \bigg(\delta \widetilde{G}_{\sigma_4 }^{s}(t)^{2}+D_{0}^{s}(t)^{2}+\int_{0}^{t}  (1+\tau)^{2\sigma_4 -1} \left(\norm{\partial_{x'}\mathbf{B}}_{2}^{2}+\norm{\partial_{\tau} \mathbf{B}}_{2}^{2}\right) \mathrm{d} \tau\bigg),\\
&\sum_{1 \leq k+\ell+2j \leq s-1}\int_{0}^{t}(1+\tau)^{2\sigma_4 }\norm{\partial_{x'}^{k} \partial_{x_n}^{\ell+2} \partial_{t}^{j} \mathbf{B}}_{2}^{2}\mathrm{d} \tau\\
\lesssim{}& \norm{\mathbf{B}_0}_{H^s}^2+\widetilde{H}_{\sigma_4 }^{s}(t)^{2}+\sigma_4 \bigg(\delta \widetilde{G}_{\sigma_4 }^{s}(t)^{2}+D_{0}^{s}(t)^{2}+\int_{0}^{t}  (1+\tau)^{2\sigma_4 -1} \left(\norm{\partial_{x'}\mathbf{B}}_{2}^{2}+\norm{\partial_{\tau} \mathbf{B}}_{2}^{2}\right) \mathrm{d} \tau\bigg).
\end{align*}
Summing these estimates together gives \eqref{3.8}, which completes the proof.
\end{proof}
To complete the energy estimate, it remains to derive bounds for the nonlinear terms \(\widetilde{N}_{\sigma_2}^{s}(t)\) and \(\widetilde{H}_{\sigma_4}^{s}(t)\), which are provided in the proposition below.

\begin{proposition}\label{P3.11}
Let the regularity index satisfy \(s \geq s_0 + 2\), and suppose that \(E_0^s(t) \leq 1\) for every \(t \geq 0\). Then the inequalities below hold for all \(t \geq 0\):
  \begin{equation}\label{3.9}
    \widetilde{N}_{\sigma_2}^{s}(t)^2\lesssim\delta\widetilde{D}_{\sigma_2 }^{s}(t)^2+E_{\sigma_1}^{s}(t)^4+ \Big(\mathscr{M}_{\infty}(t)+\mathscr{M}_{1}(t)\Big)D_{\sigma_1}^{s}(t)^2,
  \end{equation}
\begin{equation}\label{3.10}
    \widetilde{H}_{\sigma_4}^{s}(t)^2\lesssim\delta\widetilde{G}_{\sigma_4 }^{s}(t)^2+E_{\sigma_1}^{s}(t)^4+\Big(E_{\sigma_1}^{s}(t)^2+ \mathscr{M}_{\infty}(t)+\mathscr{M}_{1}(t)\Big)D_{\sigma_1}^{s}(t)^2
  \end{equation}
for any \(\delta>0\).
\end{proposition}

\begin{proof}
Following the argument in \cite{Kagei-Kobayashi2005}, these bounds follow from an application of Lemmas \ref{L4.1}–\ref{L4.3}; we therefore omit the full details of the derivation.
\end{proof}

We now turn to the weighted supremum quantities \(\widetilde{M}_{\sigma_0}^3(t)\) and \(\widetilde{B}_{\sigma_5}^2(t)\), for which we state the following result.

\begin{proposition}\label{P3.12}
Let the regularity index satisfy \(s \geq s_0 + 2\), and suppose that \(E_0^s(t) \leq 1\) for every \(t \geq 0\). Then the inequalities below hold for all \(t \geq 0\):
  \begin{equation}\label{3.11}
\widetilde{M}_{\sigma_{0}}^{3}(t) \lesssim \norm{\mathbf{u}_{0}}_{1}+\norm{\mathbf{u}_{0}}_{H^{3} \times H^{2}}+E_{\sigma_{1}}^{s}(t)^{2}+D_{0}^{s}(t)^{2}+\mathscr{M}_{0}(t)^2+\mathscr{M}_{\infty}(t)^2+\mathscr{M}_{1}(t)^2,
\end{equation}
  \begin{equation}\label{3.12}
\widetilde{B}_{\sigma_{5}}^{2}(t) \lesssim \norm{\mathbf{B}_{0}}_{1}+\mathscr{M}_{0}(t)^2+\mathscr{M}_{\infty}(t)^2+\mathscr{M}_{1}(t)^2.
\end{equation}
\end{proposition}

\begin{proof}
We may assume without loss of generality that \(t \geq 1\) throughout the proof.

By the same reasoning as in the proof of the \(L^2\) estimates in Theorem 3.1 of \cite{Kagei-Kobayashi2005}, we have
\[
\norm{\partial_{x}^{\ell} \partial_{t} I_{0}(t)}_{2} \lesssim t^{-\frac{n}{4}-\frac{1}{2}}\norm{\mathbf{u}_{0}}_{1}+e^{-ct}\norm{\mathbf{u}_0}_{H^{1} \times L^{2}}
\]
for \(\ell = 0, 1\).

To bound the time derivatives \(\partial_t I_1(t)\) and \(\partial_t I_2(t)\), we first observe that
\begin{align*}
  \partial_{t} I_{1}(t)={}&\mathscr{U}(0)\operatorname{div} \begin{pmatrix} 0 \\ \mathscr{N}_1+\mathscr{N}_2 \end{pmatrix}(t)-\mathscr{U}(1) \operatorname{div} \begin{pmatrix} 0 \\ \mathscr{N}_1+\mathscr{N}_2 \end{pmatrix}(t-1)\\
  &+\int_{t-1}^{t} \partial_{t}\mathscr{U}(t-\tau)  \operatorname{div} \begin{pmatrix} 0 \\ \mathscr{N}_1+\mathscr{N}_2 \end{pmatrix}(\tau) \mathrm{d} \tau\\
  ={}&\int_{t-1}^{t} \mathscr{U}(t-\tau) \partial_{\tau} \operatorname{div} \begin{pmatrix} 0 \\ \mathscr{N}_1+\mathscr{N}_2 \end{pmatrix}(\tau) \mathrm{d} \tau,\\
\partial_{t} I_{2}(t) ={}& \mathscr{U}(1) \operatorname{div} \begin{pmatrix} 0 \\ \mathscr{N}_1+\mathscr{N}_2 \end{pmatrix}(t-1) +\int_{0}^{t-1} \partial_{t} \mathscr{U}(t-\tau) \operatorname{div} \begin{pmatrix} 0 \\ \mathscr{N}_1+\mathscr{N}_2 \end{pmatrix}(\tau) \mathrm{d} \tau.
\end{align*}

Since \(s \geq s_0 + 2\), proceeding as in the proof of Proposition \ref{P3.8}, we deduce that
\begin{equation}\label{3.13}
  \norm{\partial_{t} \partial_{x} \mathscr{N}_1(t)}_{2}+\norm{\partial_{t} \partial_{x} \mathscr{N}_2(t)}_{2} \lesssim(1+t)^{-\sigma_{0}}\left(E_{\sigma_{1}}^{s}(t)^{2}+\mathscr{M}_{\infty}(t)^2+\mathscr{M}_{1}(t)^2\right)
\end{equation}
for all \(t \geq 0\), provided that \(E_0^s(t) \leq 1\) uniformly in \(t \geq 0\).

For \(\ell = 0, 1\), combining items (i) and (ii) of Proposition \ref{P3.3} with estimate \eqref{3.13}, we obtain
\begin{align*}
\norm{\partial_{x}^{\ell} \partial_{t} I_{1}(t)}_{2} \lesssim{}& \int_{t-1}^{t}(t-\tau)^{-\frac{\ell}{2}}\left(\norm{\partial_{\tau} \operatorname{div} \mathscr{N}_1(\tau)}_{2}+\norm{\partial_{\tau} \operatorname{div} \mathscr{N}_2(\tau)}_{2} \right)\mathrm{d} \tau \\
\lesssim{}& \left(E_{\sigma_{1}}^{s}(t)^{2}+\mathscr{M}_{\infty}(t)^2+\mathscr{M}_{1}(t)^2\right)\int_{t-1}^{t}(t-\tau)^{-\frac{\ell}{2}}(1+\tau)^{-\sigma_{0}} \mathrm{d} \tau \\
\lesssim{}&(1+t)^{-\sigma_{0}}\left(E_{\sigma_{1}}^{s}(t)^{2}+\mathscr{M}_{\infty}(t)^2+\mathscr{M}_{1}(t)^2\right).
\end{align*}

Invoking items (i) and (ii) of Proposition \ref{P3.3}, item (vi) of Proposition \ref{P3.5}, Proposition \ref{P3.8} and Lemma \ref{L4.4}, we arrive at
\begin{align*}
\norm{\partial_{x}^{\ell} \partial_{t} I_{2}(t)}_{2} \lesssim{}&\norm{ \operatorname{div} \mathscr{N}_1(t-1)}_{2}+\norm{ \operatorname{div} \mathscr{N}_2(t-1)}_{2} \\
& + \int_{0}^{t-1}(t-\tau)^{-\frac{n}{4}-\frac{1}{2}}\left(\norm{ \operatorname{div} \mathscr{N}_1(\tau)}_{1}+\norm{ \operatorname{div} \mathscr{N}_2(\tau)}_{1}\right) \mathrm{d} \tau \\
& + \int_{0}^{t-1} e^{-c(t-\tau)}\left(\norm{ \operatorname{div} \mathscr{N}_1(\tau)}_{2}+\norm{ \operatorname{div} \mathscr{N}_2(\tau)}_{2}\right) \mathrm{d} \tau \\
\lesssim{}&t^{-\frac{n}{2}} \Big(  E_{\sigma_{1}}^{s}(t-1)^{2} + M_{\sigma_{0}}^{1}(t-1)^{2}+M_{\infty, \frac{n}{2}}(t-1)^{2} \\
&+ B_{\sigma_{3}}^{1}(t-1)^{2} + B_{\infty, \frac{n}{2}}(t-1)^{2} \Big) \\
& + \left(E_{\sigma_{1}}^{s}(t)^{2}+ \mathscr{M}_{0}(t)^{2} \right)\int_{0}^{t-1}(t-\tau)^{-\frac{n}{4}-\frac{1}{2}}(1+\tau)^{-\frac{n}{2}-\frac{1}{2}\left(q(1)-\varepsilon_{1}\right)}  \mathrm{d} \tau \\
& +\left(B_{\frac{n}{4}}^{0}(t)^{2}+B_{\sigma_{3}}^{1}(t)^{2}\right)\int_{0}^{t-1}(t-\tau)^{-\frac{n}{4}-\frac{1}{2}}(1+\tau)^{-\frac{n}{4}-\sigma_{3}}\mathrm{d} \tau \\
& + \left( E_{\sigma_{1}}^{s}(t)^{2} + M_{\sigma_{0}}^{1}(t)^{2}+ M_{\infty, \frac{n}{2}}(t)^{2} + B_{\sigma_{3}}^{1}(t)^{2}  +B_{\infty, \frac{n}{2}}(t)^{2}\right)\\
&\cdot\int_{0}^{t-1} e^{-c(t-\tau)}(1+\tau)^{-\frac{n}{2}}  \mathrm{d} \tau \\
\lesssim{}&(1+t)^{-\sigma_{0}}\left(E_{\sigma_{1}}^{s}(t)^{2}+ \mathscr{M}_{0}(t)^2+\mathscr{M}_{\infty}(t)^2 \right).
\end{align*}

From the foregoing estimates, we deduce that
\[
 \norm{\partial_{t} \mathbf{u}(t)}_{H^{1}}
 \lesssim(1+t)^{-\sigma_{0}}\left(\norm{\mathbf{u}_{0}}_{1}+\norm{\mathbf{u}_0}_{H^{1} \times L^{2}}+E_{\sigma_{1}}^{s}(t)^{2}+ \mathscr{M}_{0}(t)^2+\mathscr{M}_{\infty}(t)^2+ \mathscr{M}_{1}(t)^2 \right).
\]

We next estimate \(\norm{\partial_x^2 \mathbf{u}(t)}_2\) and \(\norm{\partial_{x'} \partial_x^2 \mathbf{m}(t)}_2\). For \(k = 1, 2\), it follows from Proposition \ref{P3.1} that
\[
\norm{\partial_{x'}^{k} \partial_{x} I_{0}(t)}_{2} \lesssim t^{-\sigma_{0}}\norm{\mathbf{u}_{0}}_{1}+e^{-ct}\norm{\mathbf{u}_{0}}_{H^{3} \times H^{2}}.
\]

Since \(s \geq s_0 + 2\), repeating the argument from the proof of Proposition \ref{P3.8}, we obtain
\begin{equation}\label{3.14}
  \norm{ \partial_{x} \mathscr{N}_1(\tau)}_{H^{2}}+\norm{ \partial_{x} \mathscr{N}_2(\tau)}_{H^{2}}
 \lesssim(1+t)^{-\frac{n}{2}}\left(E_{\sigma_{1}}^{s}(t)^{2}+\mathscr{M}_{\infty}(t)^2+\mathscr{M}_{1}(t)^2\right)
\end{equation}
for all \(t \geq 0\), provided that \(E_0^s(t) \leq 1\) holds uniformly in \(t \geq 0\).

For \(k = 1, 2\), item (ii) of Proposition \ref{P3.3} combined with \eqref{3.14} yields
\begin{align*}
\norm{\partial_{x'}^{k} \partial_{x} I_{1}(t)}_{2} \lesssim{}& \int_{t-1}^{t}(t-\tau)^{-\frac{1}{2}}\left(\norm{ \operatorname{div} \mathscr{N}_1(\tau)}_{H^{2}}+\norm{ \operatorname{div} \mathscr{N}_2(\tau)}_{H^{2}}\right) \mathrm{d} \tau \\
\lesssim{}&\left(E_{\sigma_{1}}^{s}(t)^{2}+\mathscr{M}_{\infty}(t)^2+\mathscr{M}_{1}(t)^2\right) \int_{t-1}^{t}(t-\tau)^{-\frac{1}{2}}(1+\tau)^{-\frac{n}{2}} \mathrm{d} \tau \\
\lesssim{}&(1+t)^{-\frac{n}{2}}\left(E_{\sigma_{1}}^{s}(t)^{2}+\mathscr{M}_{\infty}(t)^2+\mathscr{M}_{1}(t)^2\right).
\end{align*}

For \(k = 1, 2\), combining Proposition \ref{P3.1}, item (vi) of Proposition \ref{P3.5}, item (ii) of Proposition \ref{P3.8}, estimate \eqref{3.14} and Lemma \ref{L4.4}, we obtain
\begin{align*}
&\norm{\partial_{x'}^{k} \partial_{x} I_{2}(t)}_{2} \\
 \lesssim{}& \int_{0}^{t-1}(t-\tau)^{-\frac{n}{4}-q(1)}\left(\norm{ \operatorname{div} \mathscr{N}_1(\tau)}_{1}+\norm{ \operatorname{div} \mathscr{N}_2(\tau)}_{1}\right) \mathrm{d} \tau\\
&+ \int_{0}^{t-1} e^{-c(t-\tau)}\left(\norm{ \operatorname{div} \mathscr{N}_1(\tau)}_{H^{2}}+\norm{ \operatorname{div} \mathscr{N}_2(\tau)}_{H^{2}}\right) \mathrm{d} \tau \\
 \lesssim{}&\left(E_{\sigma_{1}}^{s}(t)^{2}+ \mathscr{M}_{0}(t)^2\right)\int_{0}^{t-1}(t-\tau)^{-\frac{n}{4}-q(1)}(1+\tau)^{-\frac{n}{2}-\frac{1}{2}\left(q(1)-\varepsilon_{1}\right)} \mathrm{d} \tau \\
&+\left(B_{\frac{n}{4}}^{0}(t)^{2}+B_{\sigma_{3}}^{1}(t)^{2}\right)\int_{0}^{t-1}(t-\tau)^{-\frac{n}{4}-q(1)}(1+\tau)^{-\frac{n}{4}-\sigma_3} \mathrm{d} \tau \\
&+\left(E_{\sigma_{1}}^{s}(t)^{2}+\mathscr{M}_{\infty}(t)^2+\mathscr{M}_{1}(t)^2\right) \int_{0}^{t-1} e^{-c(t-\tau)}(1+\tau)^{-\frac{n}{2}} \mathrm{d} \tau \\
 \lesssim{}&(1+t)^{-\sigma_{0}}\Big(E_{\sigma_{1}}^{s}(t)^{2}+\mathscr{M}_{0}(t)^2+\mathscr{M}_{\infty}(t)^2+\mathscr{M}_{1}(t)^2\Big),
\end{align*}
and consequently,
\begin{align*}
\norm{\partial_{x'}^{k} \partial_{x} \mathbf{u}(t)}_{2} \lesssim{}&(1+t)^{-\sigma_{0}}\Big(\norm{\mathbf{u}_{0}}_{1}+\norm{\mathbf{u}_{0}}_{H^{3} \times H^{2}}+E_{\sigma_{1}}^{s}(t)^{2}+\mathscr{M}_{0}(t)^2+\mathscr{M}_{\infty}(t)^2+\mathscr{M}_{1}(t)^2\Big)
\end{align*}
for \(k = 1, 2\).

We now turn to the normal second-order derivative \(\partial_{x'}^k \partial_{x_n}^2 \mathbf{m}(t)\) for \(k = 0, 1\).
Proceeding analogously to the argument in \cite{Kagei-Kobayashi2005} and combining the technique from item (iv) of Proposition \ref{P2.1} with estimate \eqref{3.14}, we derive
\begin{align*}
\norm{ \partial_{x'}^{k} \partial_{x_{n}}^{2} \mathbf{m}(t)}_{2} \lesssim{} & \norm{\partial_{x'}^{k} \partial_{t} \mathbf{m}(t)}_{2}+\norm{\partial_{x'}^{k+1} \partial_{x} \mathbf{m}(t)}_{2} +\norm{\partial_{x'}^{k} \partial_{x} {a}(t)}_{2}\\
&+\norm{\partial_{x'}^{k} \operatorname{div} \mathcal{N}_1(t)}_{2}+\norm{\partial_{x'}^{k} \operatorname{div} \mathcal{N}_2(t)}_{2} \\
\lesssim{}&(1+t)^{-\sigma_{0}}\Big(\norm{\mathbf{u}_{0}}_{1}+\norm{\mathbf{u}_{0}}_{H^{3} \times H^{2}}+E_{\sigma_{1}}^{s}(t)^{2}+D_{0}^{s}(t)^{2}\\
&+\mathscr{M}_{0}(t)^2+\mathscr{M}_{\infty}(t)^2+\mathscr{M}_{1}(t)^2\Big)
\end{align*}
for \(k = 0, 1\).

It remains to bound the normal second-order derivative \(\partial_{x_n}^2 a(t)\).
Following the same strategy as in \cite{Kagei-Kobayashi2005} and combining the argument from item (iii) of Proposition \ref{P2.1} with \eqref{3.14} and Lemma \ref{L4.4}, we deduce that
\begin{align*}
\norm{\partial_{x_{n}}^{2} {a}(t)}_{2} \lesssim{}& e^{-a_{0}(t-1)}\norm{\partial_{x_{n}}^{2} {a}(1)}_{2} + \int_{1}^{t} e^{-a_{0}(t-\tau)}\Big(\norm{\partial_{x'} \partial_{x}^{2} \mathbf{m}}_{2}+\norm{\partial_{x_{n}}\partial_{\tau} \mathbf{m}}_{2}\\
& +\norm{ \partial_{x_{n}}\operatorname{div} \mathcal{N}_1}_{2}+\norm{ \partial_{x_{n}}\operatorname{div} \mathcal{N}_2}_{2}\Big) \mathrm{d} \tau \\
\lesssim{}& E_{0}^{s}(1) e^{-a_{0}(t-1)}+ \Big(\norm{\mathbf{u}_{0}}_{1}+\norm{\mathbf{u}_{0}}_{H^{3} \times H^{2}}+E_{\sigma_{1}}^{s}(t)^{2}+D_{0}^{s}(t)^{2}\\
&+\mathscr{M}_{0}(t)^2+\mathscr{M}_{\infty}(t)^2+\mathscr{M}_{1}(t)^2\Big)
\int_{1}^{t} e^{-a_{0}(t-\tau)}(1+\tau)^{-\sigma_{0}} \mathrm{d} \tau \\
\lesssim{}&(1+t)^{-\sigma_{0}}\Big(\norm{\mathbf{u}_{0}}_{1}+\norm{\mathbf{u}_{0}}_{H^{3} \times H^{2}}+E_{\sigma_{1}}^{s}(t)^{2}+D_{0}^{s}(t)^{2}\\
&+\mathscr{M}_{0}(t)^2+\mathscr{M}_{\infty}(t)^2+\mathscr{M}_{1}(t)^2\Big)
\end{align*}
for \(k = 0, 1\), where \(a_{0}=\frac{1}{\alpha+\beta}\).

Collecting all the foregoing estimates yields the desired bound \eqref{3.11} for \(\widetilde{M}_{\sigma_0}^3(t)\).

We now treat the magnetic field quantities. Invoking item (i) of Proposition \ref{P3.4}, we arrive at
\[
\norm{\partial_{t} II_{0}(t)}_{2} \lesssim t^{-\frac{n}{4}-{1}}\norm{\mathbf{B}_{0}}_{1}.
\]

To bound the time derivatives \(\partial_t II_1(t)\) and \(\partial_t II_2(t)\), we observe that
\begin{align*}
  \partial_{t} II_{1}(t)
  ={}&\int_{t-1}^{t} \mathscr{B}(t-\tau) \partial_{\tau} \operatorname{div} \mathscr{N}_3 (\tau) \mathrm{d} \tau,\\
\partial_{t} II_{2}(t) ={}& \mathscr{B}(1) \operatorname{div} \mathscr{N}_3 (t-1) +\int_{0}^{t-1} \partial_{t} \mathscr{B}(t-\tau) \operatorname{div} \mathscr{N}_3 (\tau) \mathrm{d} \tau.
\end{align*}

Since \(s \geq s_0 + 2\), proceeding as in the proof of Proposition \ref{P3.6}, we obtain
\begin{equation}\label{3.15}
  \norm{\partial_{t}  \mathscr{N}_3(t)}_{2} \lesssim(1+t)^{-\frac{n}{4}-1}\left(\mathscr{M}_{\infty}(t)^2+\mathscr{M}_{1}(t)^2\right),
\end{equation}
\begin{equation}\label{3.16}
  \norm{ \partial_{x}\mathscr{N}_3(t)}_{2} \lesssim(1+t)^{-\frac{3n}{4}} \left(\mathscr{M}_{0}(t)^2 +\mathscr{M}_{\infty}(t)^2 \right),
\end{equation}
\begin{equation}\label{3.17}
  \norm{\partial_{x}^2  \mathscr{N}_3(t)}_{2} \lesssim(1+t)^{-\frac{n}{4}-1}\left(\mathscr{M}_{\infty}(t)^2+\mathscr{M}_{1}(t)^2\right)
\end{equation}
for all \(t \geq 0\), provided that \(E_0^s(t) \leq 1\) for all \(t \geq 0\).

Exploiting the boundary condition \(\mathscr{N}_3\big|_{x_n=0} = \mathbf{0}\) together with item (ii) of Proposition \ref{P3.4} and estimate \eqref{3.15}, we deduce that
\begin{align*}
\norm{\partial_{t} II_{1}(t)}_{2} \lesssim{}& \int_{t-1}^{t}(t-\tau)^{-\frac{1}{2}}\norm{\partial_{\tau} \mathscr{N}_3(\tau)}_{2}\mathrm{d} \tau \\
\lesssim{}&\left(\mathscr{M}_{\infty}(t)^2+\mathscr{M}_{1}(t)^2\right) \int_{t-1}^{t}(t-\tau)^{-\frac{1}{2}}(1+\tau)^{-\frac{n}{4}-1} \mathrm{d} \tau \\
\lesssim{}&(1+t)^{-\frac{n}{4}-1}\left(\mathscr{M}_{\infty}(t)^2+\mathscr{M}_{1}(t)^2\right).
\end{align*}

For \(n \geq 3\), combining item (i) of Proposition \ref{P3.4}, item (iii) of Proposition \ref{P3.6}, estimate \eqref{3.16} and Lemma \ref{L4.4}, we obtain
\begin{align*}
\norm{\partial_{t} II_{2}(t)}_{2} \lesssim{}&\norm{ \operatorname{div} \mathscr{N}_3(t-1)}_{2} + \int_{0}^{t-1}(t-\tau)^{-\frac{n}{4}-1}\norm{ \operatorname{div}\mathscr{N}_3(\tau)}_{1}\mathrm{d} \tau \\
\lesssim{}&t^{-\frac{3n}{4}} \left(\mathscr{M}_{0}(t-1)^2 +\mathscr{M}_{\infty}(t-1)^2 \right)\\
& + \mathscr{M}_{0}(t)^2 \int_{0}^{t-1}(t-\tau)^{-\frac{n}{4}-1}(1+\tau)^{-\frac{n}{2}-q(1)}  \mathrm{d} \tau \\
\lesssim{}&(1+t)^{-\frac{n}{4}-1}\left(\mathscr{M}_{0}(t)^2+\mathscr{M}_{\infty}(t)^2 \right).
\end{align*}

In the two-dimensional case \(n = 2\), we similarly obtain
\begin{equation}\label{3.18}
  \norm{\mathscr{N}_3(t)}_{2}\lesssim (1+t)^{-\frac{3}{2}}\left(\mathscr{M}_{0}(t)^2+\mathscr{M}_{\infty}(t)^2 \right)
\end{equation}
for all \(t \geq 0\), provided that \(E_0^s(t) \leq 1\) for all \(t \geq 0\).

For \(n = 2\), exploiting the boundary condition \(\mathscr{N}_3\big|_{x_n=0} = \mathbf{0}\) together with item (ii) of Proposition \ref{P3.4}, estimates \eqref{3.16}, \eqref{3.18} and Lemma \ref{L4.4}, we arrive at
\begin{align*}
\norm{\partial_{t} II_{2}(t)}_{2} \lesssim{}&\norm{ \operatorname{div} \mathscr{N}_3(t-1)}_{2} + \int_{0}^{t-1}(t-\tau)^{-\frac{3}{2}}\norm{\mathscr{N}_3(\tau)}_{2}\mathrm{d} \tau \\
\lesssim{}&t^{-\frac{3}{2}} \left(\mathscr{M}_{0}(t-1)^2 +\mathscr{M}_{\infty}(t-1)^2 \right)\\
& + \left(\mathscr{M}_{0}(t)^2 +\mathscr{M}_{\infty}(t)^2 \right)\int_{0}^{t-1}(t-\tau)^{-\frac{3}{2}}(1+\tau)^{-\frac{3}{2}}  \mathrm{d} \tau \\
\lesssim{}&(1+t)^{-\frac{3}{2}}\left(\mathscr{M}_{0}(t)^2+\mathscr{M}_{\infty}(t)^2 \right).
\end{align*}

From the above estimates, we conclude that
\[
 \norm{\partial_{t} \mathbf{B}(t)}_{2}
 \lesssim(1+t)^{-\frac{n}{4}-1}\left(\norm{\mathbf{B}_{0}}_{1}+ \mathscr{M}_{0}(t)^2+\mathscr{M}_{\infty}(t)^2 +\mathscr{M}_{1}(t)^2\right).
\]

To bound the normal second-order derivative \(\partial_{x_n}^2 \mathbf{B}(t)\), we first treat the tangential component \(\partial_{x'}^2 \mathbf{B}(t)\).
Invoking item (i) of Proposition \ref{P3.4}, we obtain
\[
\norm{\partial_{x'}^2 II_{0}(t)}_{2} \lesssim t^{-\frac{n}{4}-{1}}\norm{\mathbf{B}_{0}}_{1}.
\]

Combining item (iii) of Proposition \ref{P3.4} with estimate \eqref{3.17}, we deduce that
\begin{align*}
\norm{\partial_{x'}^2 II_{1}(t)}_{2} \lesssim{}& \int_{t-1}^{t}(t-\tau)^{-\frac{1}{2}}\norm{\partial_x^2 \mathscr{N}_3(\tau)}_{2}\mathrm{d} \tau \\
\lesssim{}&\left(\mathscr{M}_{\infty}(t)^2+\mathscr{M}_{1}(t)^2\right) \int_{t-1}^{t}(t-\tau)^{-\frac{1}{2}}(1+\tau)^{-\frac{n}{4}-1} \mathrm{d} \tau \\
\lesssim{}&(1+t)^{-\frac{n}{4}-1}\left(\mathscr{M}_{\infty}(t)^2+\mathscr{M}_{1}(t)^2\right).
\end{align*}

For \(n \geq 3\), an application of item (i) of Proposition \ref{P3.4}, item (iii) of Proposition \ref{P3.6} and Lemma \ref{L4.4} yields
\begin{align*}
\norm{\partial_{x'}^2 II_{2}(t)}_{2} \lesssim{}&\int_{0}^{t-1}(t-\tau)^{-\frac{n}{4}-1}\norm{ \operatorname{div}\mathscr{N}_3(\tau)}_{1}\mathrm{d} \tau \\
\lesssim{}&\mathscr{M}_{0}(t)^2 \int_{0}^{t-1}(t-\tau)^{-\frac{n}{4}-1}(1+\tau)^{-\frac{n}{2}-q(1)}  \mathrm{d} \tau \\
\lesssim{}&(1+t)^{-\frac{n}{4}-1}\mathscr{M}_{0}(t)^2.
\end{align*}

For \(n = 2\), exploiting the boundary condition \(\mathscr{N}_3\big|_{x_n=0} = \mathbf{0}\) together with item (ii) of Proposition \ref{P3.4}, estimate \eqref{3.18} and Lemma \ref{L4.4}, we obtain
\begin{align*}
\norm{\partial_{x'}^2 II_{2}(t)}_{2} \lesssim{}& \int_{0}^{t-1}(t-\tau)^{-\frac{3}{2}}\norm{\mathscr{N}_3(\tau)}_{2}\mathrm{d} \tau \\
\lesssim{}&\left(\mathscr{M}_{0}(t)^2 +\mathscr{M}_{\infty}(t)^2 \right)\int_{0}^{t-1}(t-\tau)^{-\frac{3}{2}}(1+\tau)^{-\frac{3}{2}}  \mathrm{d} \tau \\
\lesssim{}&(1+t)^{-\frac{3}{2}}\left(\mathscr{M}_{0}(t)^2+\mathscr{M}_{\infty}(t)^2 \right).
\end{align*}

Collecting these bounds yields
\[
 \norm{\partial_{x'}^2 \mathbf{B}(t)}_{2}
 \lesssim(1+t)^{-\frac{n}{4}-1}\left(\norm{\mathbf{B}_{0}}_{1}+ \mathscr{M}_{0}(t)^2+\mathscr{M}_{\infty}(t)^2 +\mathscr{M}_{1}(t)^2\right).
\]

Finally, using the magnetic field equation and estimate \eqref{3.16}, we deduce that
\begin{align*}
 \norm{\partial_{x_n}^2 \mathbf{B}(t)}_{2}\lesssim{}&\norm{\partial_{t} \mathbf{B}(t)}_{2}+\norm{\partial_{x'}^2 \mathbf{B}(t)}_{2}+\norm{ \partial_{x}\mathscr{N}_3(t)}_{2}\\
 \lesssim{}&(1+t)^{-\frac{n}{4}-1}\left(\norm{\mathbf{B}_{0}}_{1}+ \mathscr{M}_{0}(t)^2+\mathscr{M}_{\infty}(t)^2 +\mathscr{M}_{1}(t)^2\right).
\end{align*}

Putting all these estimates together yields the desired bound \eqref{3.12} for \(\widetilde{B}_{\sigma_5}^2(t)\).
\end{proof}

Combining estimate \eqref{3.5} with inequalities \eqref{3.7}–\eqref{3.12}, we arrive at the proposition below, which directly yields the improved decay estimates for higher-order derivatives.

\begin{proposition}\label{P3.13}
Let the regularity index satisfy \(s \geq s_0 + 2\), and suppose that \(E_0^s(t) \leq 1\) for every \(t \geq 0\). If \(\norm{(\mathbf{u}_0,\mathbf{B}_0)}_{H^s} + \norm{(\mathbf{u}_0,\mathbf{B}_0)}_1 \leq \eta_1\), then the inequality below holds for all \(t \geq 0\):
\begin{equation}\label{3.19}
\mathscr{M}_{1}(t)+\widetilde{D}_{\sigma_2 }^{s}(t)+\widetilde{G}_{\sigma_4 }^{s}(t) \lesssim \eta_1.
\end{equation}
\end{proposition}
\subsubsection{Proof of Theorem \ref{T1.3} $(iii)$}\label{ss:3.3.4}
We next establish the estimate stated in item (iii) of Theorem \ref{T1.3}.
To derive refined decay rates for \(\|I_1\|_\infty\) and \(\|I_2\|_\infty\), we first collect the corresponding estimates for the nonlinearities \(\mathscr{N}_1\) and \(\mathscr{N}_2\), which are summarized in the proposition below.

\begin{proposition}\label{P3.14}
Let the regularity index satisfy \(s \geq s_0 + 2\), and suppose that \(E_0^s(t) \leq 1\) and \(a(t) \geq -\frac{1}{2}\) for every \(t \geq 0\). Then the inequalities below hold for all \(t \geq 0\):
\begin{enumerate}
\item[(i)]
\[
\norm{\mathscr{N}_{1}(t)}_{H^{\left[\frac{n}{2}\right]+1}} \lesssim (1+t)^{-\frac{n}{2}-\frac{1}{2}} \left( \mathscr{M}_{0}(t)^{2}+ \mathscr{M}_{\infty}(t)^{2}  + \mathscr{M}_{1}(t)^{2}\right).
\]
\item[(ii)]
\[
\norm{\partial_x \mathscr{N}_{1}(t)}_{1} \lesssim (1+t)^{-\frac{n}{2}-q(1)} \left( \mathscr{M}_{0}(t)^{2} + \mathscr{M}_{1}(t)^{2} \right).
\]
\item[(iii)]
\[
\norm{\mathscr{N}_{2}(t)}_{H^{\left[\frac{n}{2}\right]+1}} \lesssim (1+t)^{-\frac{n}{2}-\frac{1}{2}} \left( \mathscr{M}_{0}(t)^{2}+ \mathscr{M}_{\infty}(t)^{2}  + \mathscr{M}_{1}(t)^{2}\right).
\]
\end{enumerate}
\end{proposition}

\begin{proof}
Following the argument in \cite{Kagei-Kobayashi2005}, these estimates follow from an application of Lemmas \ref{L4.1}–\ref{L4.3} presented in the Appendix \ref{sec:4}; we therefore omit the full details.
\end{proof}

Let \(s \geq s_0 + 2\). Combining item (iii) of Proposition \ref{P3.3} with items (i) and (iii) of Proposition \ref{P3.14}, we deduce that
\begin{align*}
\norm{I_{1}(t)}_{\infty}  \lesssim {}& \int_{t-1}^{t} (t-\tau)^{-(1-\tilde{\varepsilon})}\left(\norm{\operatorname{div} \mathscr{N}_{1}(t)}_{H^{\left[\frac{n}{2}\right]}}+\norm{\operatorname{div} \mathscr{N}_{2}(t)}_{H^{\left[\frac{n}{2}\right]}}\right) \mathrm{d}\tau \\
\lesssim {}& \left( \mathscr{M}_{0}(t)^{2}+ \mathscr{M}_{\infty}(t)^{2}  + \mathscr{M}_{1}(t)^{2}\right)\int_{t-1}^{t} (t-\tau)^{-(1-\tilde{\varepsilon})} (1+\tau)^{-\frac{n}{2}-\frac{1}{2}} \mathrm{d}\tau \\
 \lesssim{}& (1+t)^{-\frac{n}{2}-\frac{1}{2}} \left( \mathscr{M}_{0}(t)^{2}+ \mathscr{M}_{\infty}(t)^{2}  + \mathscr{M}_{1}(t)^{2}\right).
\end{align*}

Invoking item (i) of Proposition \ref{P3.1}, Proposition \ref{P3.2}, items (iv), (vi) and (viii) of Proposition \ref{P3.5}, Proposition \ref{P3.14} and Lemma \ref{L4.4}, we arrive at
\begin{align*}
\norm{I_{2}(t)}_{\infty} \lesssim{}&  \int_{0}^{\frac{t-1}{2}} (t-\tau)^{-\frac{n}{2}-\frac{1}{2}} \left(\norm{\mathscr{N}_{1}(\tau)}_{1}+\norm{\mathscr{N}_{2}(\tau)}_{1}\right) \mathrm{d}\tau \\
&+ \int_{\frac{t-1}{2}}^{t-1} (t-\tau)^{-\frac{n}{2}} \left(\norm{\operatorname{div} \mathscr{N}_{1}(\tau)}_{1}+\norm{\operatorname{div} \mathscr{N}_{2}(\tau)}_{1}\right) \mathrm{d}\tau  \\
& + \int_{0}^{t-1} e^{-c(t-\tau)} \left(\norm{\mathscr{N}_{1}(\tau)}_{H^{\left[\frac{n}{2}\right]+1}}+\norm{\mathscr{N}_{2}(\tau)}_{H^{\left[\frac{n}{2}\right]+1}}\right) \mathrm{d}\tau  \\
\lesssim{}&  \left( M_{\frac{n}{4}}^{0}(t)^{2} + M_{\sigma_{0}}^{1}(t)^{2}+B_{\frac{n}{4}}^{0}(t)^{2}\right)
\int_{0}^{\frac{t-1}{2}} (t-\tau)^{-\frac{n}{2}-\frac{1}{2}} (1+\tau)^{-\frac{n}{2}}  \mathrm{d}\tau \\
&+ \left( \mathscr{M}_{0}(t)^{2} + \mathscr{M}_{1}(t)^{2} \right)\int_{\frac{t-1}{2}}^{t-1} (t-\tau)^{-\frac{n}{2}} (1+\tau)^{-\frac{n}{2}-q(1)}\mathrm{d}\tau\\
&+\left(B_{\frac{n}{4}}^{0}(t)^{2}+B_{\sigma_{3}}^{1}(t)^{2}\right)\int_{\frac{t-1}{2}}^{t-1} (t-\tau)^{-\frac{n}{2}} (1+t)^{-\frac{n}{4}-\sigma_{3}}\mathrm{d}\tau\\
&+\left( \mathscr{M}_{0}(t)^{2}+ \mathscr{M}_{\infty}(t)^{2}  + \mathscr{M}_{1}(t)^{2}\right)\int_{0}^{t-1} e^{-c(t-\tau)} (1+\tau)^{-\frac{n}{2}-\frac{1}{2}}\mathrm{d}\tau\\
\lesssim{}&(1+t)^{-\frac{n}{2}-q(1)} L(t) \left( \mathscr{M}_{0}(t)^{2}+ \mathscr{M}_{\infty}(t)^{2}  + \mathscr{M}_{1}(t)^{2}\right).
\end{align*}

Putting these estimates together, we obtain
\[
\norm{\mathbf{u}(t)-\mathscr{U}(t) \mathbf{u}_{0}}_{\infty} \lesssim (1+t)^{-\frac{n}{2}-q(1)} L(t),
\]
which is the desired decay bound. This concludes the proof of item (iii) of Theorem \ref{T1.3}.

\subsection*{Acknowledgement}
The author is deeply thankful to Professor Feng Xie, School of Mathematical Sciences, Shanghai Jiao Tong University, for valuable conversations and constructive suggestions.

\subsection*{Data Availability}
Data sharing is not applicable as no datasets were generated or analyzed during the current study.

\section*{Declarations}

\subsection*{Conflict of interest}
The authors declare that they have no conflict of interest.

\section{Appendix}\label{sec:4}

We begin by recalling the classical Gagliardo–Nirenberg interpolation inequality.

\begin{lemma}\label{L4.1}
Let \(2 \leq p \leq \infty\) and let \(j, k \in \mathbb{Z}\) satisfy
\[
0 \leq j < k, \quad k > j + n\left(\frac{1}{2} - \frac{1}{p}\right).
\]
Then
\[
\norm{\partial_x^j f}_p \lesssim \norm{f}_2^{1-a} \norm{\partial_x^k f}_2^a,
\]
where \(a = \frac{1}{k}\left(j + \frac{n}{2} - \frac{n}{p}\right)\).
\end{lemma}
\begin{proof}
See \cite{Friedman69}.
\end{proof}

To estimate the nonlinear terms, we employ the following two lemmas, whose proofs can be found in \cite{Kagei-Kobayashi2005} and are therefore omitted.

\begin{lemma}\label{L4.2}
\begin{enumerate}
\item[(i)] Let \( \mu _j\) and \(s_j\) for \(j = 1, \dots, \ell\) be  nonnegative  integers satisfying
\[
\mu_j < s_j \quad (j = 1, \dots, \ell)
\qquad\text{and}\qquad
s_1 + \cdots + s_\ell \geq (\ell - 1)s_0 + \mu,
\]
where \( \mu  =  \mu _1 + \cdots +  \mu _\ell\). Then there exist constants \(0 \leq a_j < 1\) for \(j = 1, \dots, \ell\) such that
\[
\norm{\partial_x^{ \mu _1} f_1 \cdots \partial_x^{ \mu _\ell} f_\ell}_2 \lesssim \prod_{1 \leq j \leq \ell} \norm{f_j}_2^{1-a_j} \norm{\partial_x^{s_j} f_j}_2^{a_j}.
\]

\item[(ii)] Let \(s \geq s_0\) and let \( \mu _j\) and \(s_j\) for \(j = 1, \dots, \ell\) be  nonnegative integers satisfying
\[
\mu_j \leq s_j \leq s + \mu_j \quad (j = 1, \dots, \ell)
\qquad\text{and}\qquad
s_1 + \cdots + s_\ell \geq (\ell - 1)s + \mu,
\]
where \( \mu  =  \mu _1 + \cdots +  \mu _\ell\). Then there exist constants \(0 \leq a_j \leq 1\) for \(j = 1, \dots, \ell\) such that
\[
\norm{\partial_x^{ \mu _1} f_1 \cdots \partial_x^{ \mu _\ell} f_\ell}_2 \lesssim \prod_{1 \leq j \leq \ell} \norm{f_j}_2^{1-a_j} \norm{\partial_x^{s_j} f_j}_2^{a_j}.
\]
\end{enumerate}
\end{lemma}

\begin{lemma}\label{L4.3}
\begin{enumerate}
\item[(i)] Let \(F \in C^\infty(\mathbb{R})\) and \(s \geq s_0\). Assume that  \( \mu \) and \(\nu\) are   nonnegative  integers satisfying
\[
\mu  + 2\nu \leq s.
\]
Then
\[
\begin{aligned}
&\norm{\partial_x^ \mu  \partial_t^\nu F(f)}_2 \\
\lesssim&
\begin{cases}
 \norm{ F(f)}_2 & \text{if }  \mu  = 0, \nu = 0, \\[6pt]
C_{ \mu, \nu} \left(1 + \sum_{1 \leq m \leq \nu} \norm{\partial_t^m f}_{H^{ {s}-2j}} \right)^{\nu-1} \sum_{1 \leq m \leq \nu} \norm{\partial_t^m f}_{H^{ {s}-2m}} & \text{if }  \mu  = 0, \nu \geq 1, \\[6pt]
C_{ \mu, \nu} \left(1 + \norm{\partial_x f}_{H^{s-1}} \right)^{ \mu -1} \norm{\partial_x f}_{H^{s-1}} & \text{if }  \mu  \geq 1, \nu = 0, \\[6pt]
C_{ \mu, \nu} \left(1 + \norm{\partial_x f}_{H^{s-1}} + \sum_{1 \leq m \leq \nu} \norm{\partial_t^m f}_{H^{ {s}-2m}} \right)^{ \mu +\nu-1} \\
\cdot\left( \norm{\partial_x f}_{H^{s-1}} + \sum_{1 \leq m \leq \nu} \norm{\partial_t^m f}_{H^{ {s}-2m}} \right) & \text{if }  \mu  \geq 1, \nu \geq 1,
\end{cases}
\end{aligned}
\]
where
\[
C_{ \mu, \nu} = \sup_{1 \leq k \leq  \mu +\nu} \norm{F^{(k)}(f)}_\infty.
\]

\item[(ii)] Let \(F \in C^\infty(\mathbb{R})\) and \(s \geq s_0\). Assume that \( {s}\), \(s_j\), \( \mu _j\) for \(j = 2, \dots, \ell\), \( \mu _1\) and \(\nu_1\) are nonnegative integers satisfying
\begin{align*}
   \mu _1 \leq{}& {s} - 2\nu_1,\\
 \mu _j \leq{}& s_j \leq s+\mu _j \quad (j = 2, \dots, \ell),\\
({s}-2&\nu_1)+ s_2 + \cdots + s_\ell \geq (\ell - 1)s +  \mu,
\end{align*}
where \( \mu  =  \mu _1 + \cdots +  \mu _\ell\). Then
\[
\begin{aligned}
&\norm{\partial_x^{ \mu _1} \partial_t^{\nu_1} F(f_1) \prod_{2 \leq j \leq \ell} \partial_x^{ \mu _j} f_j}_2 \\
\lesssim&
\begin{cases}
\norm{F(f_1)}_{\infty} \prod_{2 \leq j \leq \ell} \norm{f_j}_{H^{s_j}} & \text{if }  \mu _1 = 0, \nu_1 = 0, \\[6pt]
C_{ \mu _1, \nu_1} \left(1 + \sum_{1 \leq m \leq \nu_1} \norm{\partial_t^m f_1}_{H^{ {s}-2m}} \right)^{\nu_1-1} \\
 \cdot\left( \sum_{1 \leq m \leq \nu_1} \norm{\partial_t^m f_1}_{H^{ {s}-2m}} \right) \prod_{2 \leq j \leq \ell} \norm{f_j}_{H^{s_j}} & \text{if }  \mu _1 = 0, \nu_1 \geq 1, \\[6pt]
C_{ \mu _1, \nu_1} \left(1 + \norm{\partial_x f_1}_{H^{s-1}} \right)^{ \mu _1-1} \norm{\partial_x f_1}_{H^{s-1}}\prod_{2 \leq j \leq \ell} \norm{f_j}_{H^{s_j}} & \text{if }  \mu _1 \geq 1, \nu_1 = 0, \\[6pt]
C_{ \mu _1, \nu_1} \left(1 + \norm{\partial_x f_1}_{H^{s-1}} + \sum_{1 \leq m \leq \nu_1} \norm{\partial_t^m f_1}_{H^{ {s}-2m}} \right)^{ \mu _1+\nu_1-1} \\
\cdot\left( \norm{\partial_x f_1}_{H^{s-1}} + \sum_{1 \leq m \leq \nu_1} \norm{\partial_t^m f_1}_{H^{ {s}-2m}} \right)\prod_{2 \leq j \leq \ell} \norm{f_j}_{H^{s_j}} & \text{if }  \mu _1 \geq 1, \nu_1 \geq 1,
\end{cases}
\end{aligned}
\]
where
\[
C_{ \mu _1, \nu_1} = \sup_{1 \leq k \leq  \mu _1+\nu_1} \norm{F^{(k)}(f_1)}_\infty.
\]
\end{enumerate}
\end{lemma}

Finally, we record a standard elementary estimate for convolution integrals on the half-line. Its proof is straightforward and is therefore omitted.

\begin{lemma}\label{L4.4}
Let \(0 < s_1 \leq s_2\). Then
\[
\begin{aligned}
\int_{0}^{t} (1 + t - \tau)^{-s_1} (1 + \tau)^{-s_2}  \mathrm{d} \tau
&= \int_{0}^{t} (1 + t - \tau)^{-s_2} (1 + \tau)^{-s_1}  \mathrm{d} \tau \\
&\lesssim
\begin{cases}
(1 + t)^{-s_1}, & \text{if } s_2 > 1, \\
(1 + t)^{-s_1} \ln(1 + t), & \text{if } s_2 = 1, \\
(1 + t)^{1 - s_1 - s_2}, & \text{if } s_2 < 1.
\end{cases}
\end{aligned}
\]
\end{lemma}

\end{document}